\documentclass[11pt,reqno]{amsart}
\usepackage{mathtools}
\usepackage{xcolor}
\usepackage[a4paper, left = 2cm, right = 2cm, top = 1.5cm, bottom = 1.5cm]{geometry}
\usepackage{dsfont} 
\usepackage[all]{xy}
\usepackage{graphicx} 
\usepackage[
hidelinks]{hyperref}
\usepackage[bb=boondox]{mathalfa}

\usepackage{tikz}  
\usetikzlibrary{arrows.meta} 
\usetikzlibrary{shapes.geometric}

\usepackage{amssymb}

\usepackage{mdframed}

\DeclareMathAlphabet{\mathbbold}{U}{bbold}{m}{n}

\newtheorem{proposition}{Proposition}
\newtheorem{theorem}[proposition]{Theorem}
\newtheorem{lemma}[proposition]{Lemma}
\newtheorem{corollary}[proposition]{Corollary}
\theoremstyle{definition}
\newtheorem{definition}[proposition]{Definition}
\newtheorem{remark}[proposition]{Remark}
\newtheorem{example}[proposition]{Example}

\numberwithin{proposition}{section}
\numberwithin{equation}{section}

\newcommand{\cst}{\ensuremath{\mathrm{C}^\ast}}

\newcommand{\id}{\mathrm{id}}

\newcommand{\I}{\mathds{1}}

\newcommand{\op}{\text{\raisebox{1.5pt}{\scalebox{0.7}{\tiny{\rm{op}}}}}}

\newcommand{\ket}[1]{{\left|#1\right\rangle}}
\newcommand{\bra}[1]{{\left\langle#1\right|}}

\newcommand{\CC}{\mathbb{C}}

\newcommand{\RR}{\mathbb{R}}

\newcommand{\XX}{\mathbb{X}}

\newcommand{\Tr}{\mathrm{Tr}}

\DeclareMathOperator{\C}{C}

\DeclareMathOperator{\Mat}{\mathsf{Mat}}

\title{Laplace operators on quantum graphs}

\author{Arkadiusz Bochniak}
\address{Institute of Theoretical Physics, Jagiellonian University, {\L}ojasiewicza 11, Krak{\'o}w, 30-348, Poland}
\address{Max-Planck-Institut f{\"u}r Quantenoptik, Garching, Germany}
\address{Munich Center for Quantum Science and Technology, Munich, Germany}
\email{arkadiusz.bochniak@uj.edu.pl}

\author{Dawid Jasi{\'n}ski}
\address{University of Warsaw, Faculty of Physics, Department of Mathematical Methods in Physics, Poland}
\email{d.jasinski4@student.uw.edu.pl}

\author{Pawe{\l} Kasprzak}
\address{University of Warsaw, Faculty of Physics, Department of Mathematical Methods in Physics, Poland}
\email{pawel.kasprzak@fuw.edu.pl}

\keywords{Quantum graph, Laplace operator, Incidence operator}

\makeatletter
\@namedef{subjclassname@2020}{\textup{2020} Mathematics Subject Classification}
\makeatother

\begin{document}
    \maketitle

\begin{abstract}
We introduce a class of Laplace operators associated with quantum graphs in the operator-system framework. To this end, we investigate structural properties of quantum graphs related to the Schur product and transposition. Our main innovation is the identification of a specific inner product on the space of matrices with respect to which the projection onto the operator system underlying a quantum graph is orthogonal. This enables us to define a quantum analogue of the classical incidence operator and the associated Laplace operator. We compare this construction with a recently proposed alternative definition and examine the resulting Laplace operators for several families of quantum graphs. These results extend fundamental constructions from spectral graph theory to the operator-algebraic setting, providing a step toward a spectral theory of quantum graphs that generalizes classical graph-theoretic concepts while revealing new connections with quantum information theory and noncommutative geometry.
\end{abstract}

\section{Introduction}
The graph Laplacian is one of the central objects in spectral graph theory, providing a bridge between the combinatorial structure of a graph and its algebraic and spectral properties \cite{Chung,GodsilRoyle,Spielman}. Defined from the adjacency and degree matrices of a graph, the Laplacian encodes information about connectivity, diffusion processes, random walks, and graph partitioning. Its spectrum and eigenvectors reveal key structural characteristics, including the number of connected components, the number of spanning trees, and measures of graph connectivity. In many ways, the graph Laplacian serves as a discrete analogue of the classical Laplace operator from differential geometry and partial differential equations, making it a fundamental tool in both theoretical and applied studies of networks.

Quantum graphs emerged at the intersection of operator algebras and quantum information theory, motivated in particular by the study of zero-error communication through quantum channels. In the classical setting, the confusability graph of a communication channel encodes which input symbols the receiver cannot distinguish, and graph parameters such as the independence number and the Lovász theta function describe limits on zero-error information transmission. The quantum generalization of this framework led to the theory of noncommutative graphs, introduced through operator spaces associated with quantum channels by Duan, Severini, and Winter \cite{Duan_2013}. This perspective may be viewed as a quantum analogue of the adjacency matrix representation of a graph, with classical graphs recovered as diagonal operator structures.

A complementary approach, originating from the work of Weaver \cite{NW2015}, describes quantum graphs as operator systems: self-adjoint linear subspaces of operator algebras containing the identity operator. The operator-system formulation provides a natural framework for studying quantum graph homomorphisms, graph colorings, and noncommutative analogues of classical graph invariants. These two viewpoints, the operator-space approach based on quantum adjacency relations and the operator-system approach based on quantum relations, are closely connected and have led to a rich theory \cite{PSSTW,Stahlke}. A central objective of quantum graph theory is the extension of classical graph invariants to the operator setting. Quantum analogues of the independence number, clique number, chromatic number, orthogonal rank, and the Lovász theta function have been introduced and studied in relation to quantum communication capacities and semidefinite optimization \cite{Duan_2013, DuanWinter, kim2017}. These parameters often differ from their classical counterparts because quantum players may exploit entanglement measurements, leading to genuinely quantum phenomena such as violations of classical coloring constraints and improved communication protocols \cite{Duan_2013, kim2017}.

Quantum graphs are also deeply connected with the theory of quantum games and nonlocal correlations. Many graph-theoretic problems, including graph coloring, homomorphism, and isomorphism questions, admit formulations as synchronous nonlocal games, where players attempt to satisfy graph constraints using classical, quantum, or commuting-operator strategies. Quantum graph homomorphisms can be interpreted as the existence of perfect strategies for suitable nonlocal games, establishing a direct connection between operator-algebraic graph theory and quantum information-theoretic tasks \cite{mancinska2016, roberson2015}. This connection has revealed that quantum graph parameters are not merely abstract generalizations of classical invariants, but quantities with operational meanings in terms of entanglement-assisted protocols.

In this paper, we investigate the role of Laplace-type operators in the setting of quantum graphs and study how classical notions from spectral graph theory extend to noncommutative frameworks. Motivated by its role in classical graph theory, we explore possible constructions and properties of Laplace operators associated with quantum graphs arising from operator-algebraic approaches. In particular, we study the interplay between operator-algebraic structures underlying quantum graphs and analytic properties of their associated Laplace operators, with the broader goal of developing a spectral framework for quantum graphs analogous to the classical theory of graph Laplacians. Such a framework may provide further connections between quantum graph theory, quantum information theory, and noncommutative geometry, as well as contribute to the understanding of quantum graph parameters and their operational interpretations in quantum communication and nonlocality.

The paper is organized as follows. In Section~\ref{sec:class}, we begin with a review of classical graph theory, with particular emphasis on incidence matrices and the graph Laplace operator, which provide the main motivation for our subsequent constructions. In Section~\ref{sec:quantum}, we introduce the basic concepts of quantum graph theory and summarize the main ideas underlying the operator-algebraic approach to quantum graphs. Sections~\ref{sec:Schur} and \ref{sec:transpositions} are devoted to the study of specific properties of the Schur product and transposition operations for quantum graphs. We then investigate edge spaces associated with quantum graphs in Section~\ref{sec:edge} and introduce Hilbert space structures for which the corresponding projections onto the operator systems defining quantum graphs become orthogonal. Building on these constructions, we introduce a quantum analogue of the incidence operator in Section~\ref{sec:incidence}. In Section~\ref{sec:matsuda}, we compare this construction with a related notion introduced in \cite{Matsuda_2024}. Finally, in Section~\ref{sec:examples}, we study properties of Laplace operators associated with families of quantum graphs introduced in \cite{Kiefer_2026}. 

\section{Classical graph theory}
\label{sec:class}

In this section, we review the framework of finite classical graphs. In particular, we recall the standard definitions of directed and undirected graphs, along with their associated adjacency and incidence matrices. By interpreting these concepts in operator-algebraic terms, we lay the groundwork for their quantum analogues.

All graphs are assumed to be finite. We will denote by $\Tr$ the unnormalized trace. The set $\{1,2,\ldots n\}$ will be denoted by $[n]$. The cardinality of a finite set $X$ will be denoted by $|X|$, e.g., $|[n]|=n$.  

\begin{definition}
    A finite classical graph is a pair $G=(V, E)$, where the finite set $V$ is referred to as a set of vertices, and $E \subseteq V \times V$ will be referred to as a set of edges. The source and target maps are denoted by $s, t: E\to V$ respectively, where $s(i,j) = i$ and $t(i,j) = j$. We will typically write $i \sim j$ to denote $(i, j) \in E$.
    
    A graph $G$ is called \emph{directed} if $i \sim j$ implies $j \nsim i$ for all $i, j \in V$. Conversely, $G$ is \emph{undirected} if $i \sim j$ implies $j \sim i$ for all $i, j \in V$. 
    
    If $G$ is an undirected graph, we typically identify the symmetric pairs $(i, j)$ and $(j, i)$ with a single undirected edge $\{i, j\}$. Every undirected graph can be upgraded to a directed one by an arbitrary choice of orientation for every edge.  

    If $(i, i) \in E$, then we say that $G$ has a loop at the vertex $i$. 
\end{definition}

\begin{definition}
    The adjacency matrix of a graph $G=(V, E)$ is a matrix $A \in \Mat_{|V|}(\mathbb{C})$ defined by
    $$
    A_{ij} = \begin{cases}
        1 & \text{if } (j, i) \in E \text{ (i.e., there is an edge from } j \text{ to } i\text{)},\\
        0 & \text{otherwise}.
    \end{cases}
    $$
    
    Moreover, for a directed graph, we define the incidence matrix $K \in \Mat_{|E| \times |V|}(\mathbb{C})$ as
    
    $$K_{e,v} = \begin{cases}
        1 & \text{if } v = t(e),\\
        -1 & \text{if } v = s(e),\\
        0 & \text{otherwise}.
    \end{cases}
    $$
    The incidence matrix of an undirected graph is not uniquely defined. However, we can fix one by choosing an arbitrary orientation of edges. Moreover, if $e$ is a loop, we set $K_{e,v}=0$ for every $v \in V$. 
\end{definition}

\begin{remark}
    Let $K$ be the incidence matrix of an undirected graph $G=(V, E)$, induced by a choice of orientation. For any diagonal matrix $O \in D_{|E|}(\mathbb{C})$ with entries restricted to $\pm 1$ on the diagonal, the matrix $OK$ yields the incidence matrix of the same graph under a different orientation. Specifically, if $O_{e,e} = -1$, the matrix $O$ reverses the orientation of the edge $e$. 
\end{remark}

For the needs of the next definition, let us denote the regular representation of the algebra $\C(V)$ by $L$: given $x\in \C(V)$, the operator $L_x:\C(V)\to \C(V)$ is defined by $L_x(y)=xy$. If $V=[n]$ then, identifying $\C(V)$ with $\mathbb{C}^n$, we can view $L_x$ as the diagonal matrix $L_x=\textrm{diag}(x(1),x(2),\ldots, x(n))$. In particular,  viewing the adjacency matrix as a linear operator on $\mathbb{C}^V$, we can consider commutators $[L_x,A]\in\Mat_n(\mathbb{C})$. This yields the following definition. 
\begin{definition}
For a classical graph $G=(V,E)$ with an adjacency matrix $A$, we define the following {\it incidence} map
\begin{align*}
    \widetilde{K}(x)=[L_x,A], \quad x\in \C(V).
\end{align*}
\end{definition}
To understand the relationship between this transformation and the classical incidence matrices, we first introduce the notion of edge spaces, offering the natural codomain for incidence maps. Given a directed graph, the edge space $S_A$ is defined as the linear span of the standard basis $\{E_{ij} : (j,i) \in E\}$ in the space of matrices. However, for an undirected graph, $A$ is symmetric and thus both $E_{ij}$ and $E_{ji}$ would belong to such $S_A$. In this case, the incidence map $\widetilde{K}$ naturally targets the antisymmetric subspace $S_A^a$. To define a basis for $S_A^a$, we must fix an arbitrary orientation of the graph, choosing exactly one directed edge for each undirected pair. If $E^+ \subset E$ is this chosen set of oriented edges, then $S_A^a$ has a natural basis $\{E_{ij} - E_{ji} : (j,i) \in E^+\}$.

\begin{proposition}
    If $G$ is an undirected graph with a chosen orientation $E^+$, the image of the map $\widetilde{K}$ is contained in the subspace $S_A^a$. Moreover, the coordinate matrix of the operator $\widetilde{K}$ in the basis $\{E_{ij} - E_{ji} : (j,i) \in E^+\}$ is equal to the classical incidence matrix $K$ under the same orientation.
    Similarly, if $G$ is a directed graph, the image of $\widetilde{K}$ lies in $S_A$, and its coordinate matrix in the basis $\{E_{ij} : (j,i) \in E\}$ is exactly equal to $K$.
\end{proposition}

\begin{proof}
    Since $S_A$ is a bimodule generated by $A$, we immediately obtain $\widetilde{K}(\C(V))\subseteq S_A$. For any $x\in \C(V)$:
    $$
    \widetilde{K}(x) = L_x A - A L_x = \sum_{(j,i)\in E} (x_i-x_j)E_{ij}.
    $$
    For an undirected graph, grouping symmetric edges yields:
    $$
    \widetilde{K}(x) = \sum_{(j,i)\in E^+} (x_i-x_j)(E_{ij}-E_{ji}).
    $$
    Since the $e$-th row of the classical incidence matrix $K$ computes $(Kx)_e = x_{t(e)} - x_{s(e)} = x_i - x_j$, the coordinates of $\widetilde{K}(x)$ exactly match the rows of $K$.
\end{proof}

\begin{example}
    \begin{figure}[h!tbp]
    \centering
    
    \begin{minipage}{0.45\textwidth}
        \centering
        \begin{tikzpicture}[
            every node/.style={circle, draw, minimum size=8mm, font=\bfseries}
        ]
            \node (1) at (0, 2) {1};
            \node (2) at (-1.5, 0) {2};
            \node (3) at (1.5, 0) {3};
            
            \draw[thick] (1) -- (2);
            \draw[thick] (2) -- (3);
            \draw[thick] (3) -- (1);
        \end{tikzpicture}
        \caption{Undirected graph $G_1$}
        \label{fig:graph_undirected}
    \end{minipage}\hfill
    \begin{minipage}{0.45\textwidth}
        \centering
        \begin{tikzpicture}[
            every node/.style={circle, draw, minimum size=8mm, font=\bfseries},
            >={Stealth[scale=1.2]} 
        ]
            \node (1) at (0, 2) {1};
            \node (2) at (-1.5, 0) {2};
            \node (3) at (1.5, 0) {3};
            
            \draw[thick, ->] (1) -- (2);
            \draw[thick, ->] (2) -- (3);
            \draw[thick, ->] (3) -- (1);
        \end{tikzpicture}
        \caption{Directed graph $G_2$}
        \label{fig:graph_directed}
    \end{minipage}
    
\end{figure}
Let $G_1$ and $G_2$ be the graphs presented in Figures \ref{fig:graph_undirected} and \ref{fig:graph_directed}. We denote their adjacency matrices, classical incidence matrices, and incidence maps by $A_i, K_i$, and $\widetilde{K}_i$, respectively, where $i\in \{1,2\}$. We assume the orientation $1 \to 2, 2 \to 3, 3 \to 1$ for both graphs. Since $A_{ij}=1$ denotes an edge from $j$ to $i$, we have:
$$
A_1=\left[\begin{array}{ccc}
    0 & 1 & 1 \\
    1 & 0 & 1 \\
    1 & 1 & 0 \\
\end{array}\right], \quad A_2=\left[\begin{array}{ccc}
    0 & 0 & 1 \\
    1 & 0 & 0 \\
    0 & 1 & 0 \\
\end{array}\right]
$$
and their classical incidence matrices are:
$$
K_1=K_2=\left[\begin{array}{ccc}
    -1 & 1 & 0 \\
    0 & -1 & 1 \\
    1 & 0 & -1 \\
\end{array}\right].
$$
Now we compute the algebraic incidence maps:
$$ \begin{aligned}
\widetilde{K}_1(x) &= L_xA_1 - A_1L_x \\
&= \left[\begin{array}{ccc}
    0 & x_1-x_2 & x_1-x_3 \\
    x_2-x_1 & 0 & x_2-x_3 \\
    x_3-x_1 & x_3-x_2 & 0 \\
\end{array}\right] \\
&= (x_2-x_1)(E_{21}-E_{12}) + (x_3-x_2)(E_{32}-E_{23}) + (x_1-x_3)(E_{13}-E_{31})
\end{aligned} $$
and
$$ \begin{aligned}
\widetilde{K}_2(x) &= L_xA_2 - A_2L_x \\
&= \left[\begin{array}{ccc}
    0 & 0 & x_1-x_3 \\
    x_2-x_1 & 0 & 0 \\
    0 & x_3-x_2 & 0 \\
\end{array}\right] \\
&= (x_2-x_1)E_{21} + (x_3-x_2)E_{32} + (x_1-x_3)E_{13}
\end{aligned} $$

By observing the coordinates in the respective bases, ($(E_{ij}-E_{ji})_{(j,i)\in E^+}$ for $G_1$ and $(E_{ij})_{(j,i)\in E}$ for $G_2$) we see they yield $(x_2-x_1, x_3-x_2, x_1-x_3)^T$. This perfectly matches the standard multiplication $K_{1}x$ and $K_{2}x$. 

From now on, we will write $K$ instead of $\widetilde{K}$.
\end{example}
  
Another feature of graphs defined using algebraic methods is the number of edges. The first attempt to do this may be to state that the number of edges is equal to the dimension of the edge space. Nevertheless, the previous example showed us that the dimension of the edge space differs for a given graph depending on the oriented or non-oriented context, and thus this number should be defined depending on it. 
\begin{proposition}
    The number of edges $|E|$ of directed (undirected) graph $G=(V, E)$ is given by any of the following formulas
    \begin{itemize}
        \item $|E|=\textrm{dim}(S_A)$ ($|E|=\textrm{dim}(S_A^a)$)
        \item $|E|=\langle \Omega, A \Omega \rangle$ ($|E|=\frac{1}{2}\langle \Omega, A \Omega \rangle$), where $\langle\cdot, \cdot \rangle$ is canonical (unnormalized) scalar product and $\Omega=[1, \ldots, 1]^T$.
        \item $|E|=\Tr(D)$ ($|E|=\frac{1}{2}\Tr(D)$), where $D$ is the degree matrix of $G$.
    \end{itemize}
\end{proposition}

In general, a graph can have undirected edges $((i,j) ,  (j,i)\in E)$ directed edges, and loops. Their number is coherent with the following definition. 
\begin{definition}\label{loops-directed-undirected}
    Let $A$ be an adjacency matrix of the graph $G=(V, E)$. 
    \begin{itemize}
        \item The number of loops is given by $|E_l|:=\Tr(A)$.
        \item The number of undirected edges is given by $|E_u|:=\frac{1}{2}\Tr(A(A-\id))$.
        \item The number of directed edges is given by $|E_d|:=\Tr(A(A^\dagger-A))$
    \end{itemize}
\end{definition}
Indeed, observe that $\Tr(A)$ enumerates the non-zero diagonal entries, which correspond precisely to self-loops. Meanwhile, $\Tr(A^2)$ yields the number of pairs $(i,j)$ for which both $A_{ij}=1$ and $A_{ji}=1$, thereby accounting for both loops and symmetric (undirected) edges. Finally, since $A$ is a binary matrix, $\Tr(AA^\dagger)$ represents the total number of ones in $A$. Now, the central object of our paper. 
\begin{definition}
    Let $G=(V, E)$ be a graph. The Laplacian of a graph is defined by $\Delta:=K^\dagger K$, where the adjoint ${}^\dagger$ is taken with respect to the Hilbert-Schmidt scalar product on $S_A$.
\end{definition}
\begin{proposition}Let $G$ be a graph and $A$ the adjacency matrix and $\Delta$ the Laplacian.  The following formula holds
    $\Delta=L_{(A+A^T)\Omega}-(A+A^T)$.
\end{proposition}
\begin{proof}
    First of all viewing $K:\C(V)\to S_A$ as a map between Hilbert spaces (with the Hilbert-Schmidt scalar product on  $S_A\subset\Mat_{|V|}(\mathbb{C})$), we have    $K^\dagger(X)=(X-X^T)\Omega$, where $X\in S_A$. Indeed, observe first that for $A\in M_n(\mathbb{C})$ and $|v\rangle\langle w^*|=\bigl(v_i\overline{w_j}\bigr)_{i,j=1}^n$, where $v,w\in\mathbb{C}^n$, we have $A\cdot_S |v\rangle\langle w^*|= L_vAL_w$, where ${}\cdot_S{}$ denotes the Schur product of matrices. To see this, observe that the $(i,j)$-entry of the Schur product is
\[
\bigl(A\cdot_S |v\rangle\langle w^*|\bigr)_{ij}
=
a_{ij}\,(|v\rangle\langle w^*|)_{ij}
=
a_{ij}v_i\overline{w_j}.
\]
On the other hand,
\[
(L_vAL_w)_{ij}
=
\sum_{k,l=1}^n
(L_v)_{ik}A_{kl}(L_w)_{l j}.
\]
Since $(L_v)_{ik}=v_i\delta_{ik}$ and $(L_w)_{l j}=\overline{w_j}\delta_{l j}$, we get
\[
(L_vAL_w)_{ij}
=
a_{ij}v_i\overline{w_j}
=
\bigl(A\cdot_S |v\rangle\langle w^*|\bigr)_{ij}
\]
for every $1\leq i,j\leq n$. Since the two matrices agree entrywise, the claimed identity follows. Using this property, we have  
    \begin{align*}
        \langle X, K(v) \rangle_{HS}&=\langle X, L_v A-A L_v \rangle_{HS}\\
        &=\langle X, A\cdot_S(\ket{v}\bra{ \Omega}-\ket{\Omega}\bra{v^*})\rangle_{HS}\\
        &=\langle X, (\ket{v}\bra{\Omega}-\ket{\Omega}\bra{v^*})\rangle_{HS}\\
        &=\Tr(X^\dagger \ket{v}\bra{\Omega})-\Tr(X^\dagger \ket{\Omega}\bra{v^*})\\
        &=\Tr(X^\dagger \ket{v}\bra{\Omega})-\Tr((X^T)^\dagger \ket{v}\bra{\Omega})\\
        &=\langle \Omega, X^\dagger v \rangle - \langle \Omega, (X^T)^\dagger v \rangle\\
        &=\langle (X-X^T)\Omega, v\rangle.
    \end{align*}

    Since $(L_x A)^T=A^T L_x$ for all $x\in \C(V)$, we have
    \begin{align*}
        K^\dagger K(x)&=K^\dagger(L_x A-A L_x)=(L_{x}(A+A^T)-(A+A^T)L_x)\Omega\\
        &=(L_{(A+A^T)\Omega}-(A+A^T))x.
    \end{align*}
\end{proof}
Note that if $G$ is directed graph, then $\Delta=D-(A+A^T)$ and if $G$ is undirected, then $\Delta=2D-2A$. Moreover, the Laplacian is invariant under adding a loop to it.
\begin{example}\label{classical incidence matrix}
    Let us consider a mixed graph $G_3$, which simultaneously contains a loop, a directed edge, and an undirected edge - see Figure~\ref{fig:graph_mixed}.
    
    \begin{figure}[h!tbp]
        \centering
        \begin{tikzpicture}[
            every node/.style={circle, draw, minimum size=8mm, font=\bfseries},
            >={Stealth[scale=1.2]}
        ]
            \node (1) at (0, 2) {1};
            \node (2) at (-1.5, 0) {2};
            \node (3) at (1.5, 0) {3};
            
            \draw[thick, ->] (1) to [out=135, in=45, looseness=8] (1); 
            \draw[thick, ->] (1) -- (2); 
            \draw[thick] (2) -- (3); 
        \end{tikzpicture}
        \caption{Mixed graph $G_3$}
        \label{fig:graph_mixed}
    \end{figure}
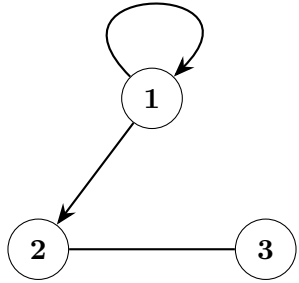

    The adjacency matrix $A$ for the graph $G_3$ takes the following form:
    \[A = \left[\begin{array}{ccc}
        1 & 1 & 0 \\
        0 & 0 & 1 \\
        0 & 1 & 0 \\
    \end{array}\right].\]
     Definition \ref{loops-directed-undirected} yields $|E_l|=|E_d|=|E_u|=1$ which is a obviously correct. 
    The incidence map $\widetilde{K}_3$ for this  graph is computed as follows
    \begin{align*}
        \widetilde{K}_3(x) &= L_x A - A L_x \\
        &= \left[\begin{array}{ccc}
            x_1 & 0 & 0 \\
            0 & x_2 & 0 \\
            0 & 0 & x_3 \\
        \end{array}\right] \left[\begin{array}{ccc}
            1 & 1 & 0 \\
            0 & 0 & 1 \\
            0 & 1 & 0 \\
        \end{array}\right] - \left[\begin{array}{ccc}
            1 & 1 & 0 \\
            0 & 0 & 1 \\
            0 & 1 & 0 \\
        \end{array}\right] \left[\begin{array}{ccc}
            x_1 & 0 & 0 \\
            0 & x_2 & 0 \\
            0 & 0 & x_3 \\
        \end{array}\right] \\
        &= \left[\begin{array}{ccc}
            0 & x_1 - x_2 & 0 \\
            0 & 0 & x_2 - x_3 \\
            0 & x_3 - x_2 & 0 \\
        \end{array}\right].
    \end{align*}
    Expressed in the standard matrix basis, this yields:
    \[\widetilde{K}_3(x) = (x_1 - x_2)E_{12} + (x_2 - x_3)E_{23} + (x_3 - x_2)E_{32}.\]
    Finally, we compute the Laplacian $\Delta$ of the graph $G_3$ using the formula $\Delta = L_{(A+A^T)\Omega} - (A+A^T)$. First, we find the symmetric part $A + A^T$:
    \[A + A^T = \left[\begin{array}{ccc}
        1 & 1 & 0 \\
        0 & 0 & 1 \\
        0 & 1 & 0 \\
    \end{array}\right] + \left[\begin{array}{ccc}
        1 & 0 & 0 \\
        1 & 0 & 1 \\
        0 & 1 & 0 \\
    \end{array}\right] = \left[\begin{array}{ccc}
        2 & 1 & 0 \\
        1 & 0 & 2 \\
        0 & 2 & 0 \\
    \end{array}\right].\]
    Next, we apply it to the vector $\Omega = [1, 1, 1]^T$ to obtain the diagonal entries of $L_{(A+A^T)\Omega}$:
    \[(A + A^T)\Omega = \left[\begin{array}{ccc}
        2 & 1 & 0 \\
        1 & 0 & 2 \\
        0 & 2 & 0 \\
    \end{array}\right] \left[\begin{array}{c} 1 \\ 1 \\ 1 \end{array}\right] = \left[\begin{array}{c} 3 \\ 3 \\ 2 \end{array}\right] \implies L_{(A+A^T)\Omega} = \left[\begin{array}{ccc}
        3 & 0 & 0 \\
        0 & 3 & 0 \\
        0 & 0 & 2 \\
    \end{array}\right].\]
    Subtracting $A + A^T$ from this diagonal matrix yields the Laplacian:
    \[\Delta = \left[\begin{array}{ccc}
        3 & 0 & 0 \\
        0 & 3 & 0 \\
        0 & 0 & 2 \\
    \end{array}\right] - \left[\begin{array}{ccc}
        2 & 1 & 0 \\
        1 & 0 & 2 \\
        0 & 2 & 0 \\
    \end{array}\right] = \left[\begin{array}{ccc}
        1 & -1 & 0 \\
        -1 & 3 & -2 \\
        0 & -2 & 2 \\
    \end{array}\right].\]
    Again, we observe that the Laplacian is symmetric and its rows (and columns) sum to zero, conforming to the standard properties of graph Laplacians, regardless of the loop present in the original adjacency matrix. Moreover, off-diagonal elements of the Laplace matrix are 0 if corresponding vertices are not connected, $-1$ if they are connected by a directed edge, and $-2$ if they are connected by an undirected edge. Incidentally, let us observe that the Laplace matrix is not sensitive to the orientation of directed edges.
\end{example}

\section{Quantum Graph Theory}
\label{sec:quantum}

The transition from classical to quantum mathematics relies on replacing commutative algebras by non-commutative ones. The foundational justification for this shift is the celebrated Gelfand--Naimark theorem, see e.g. \cite[Thm 11.18]{rudin1991}:

\begin{theorem}[Gelfand--Naimark]
    The category of compact Hausdorff spaces $\mathcal{CS}$ is contravariantly equivalent to the category of commutative unital $\cst$-algebras $\mathcal{CC}^*$ (i.e., $\mathcal{CS} \simeq \mathcal{CC}^*{}^{\textrm{\op}}$).
\end{theorem}

The commutative $\cst$-algebra associated with a space $X$ is denoted by $\C(X)$. Motivated by this equivalence, we can drop with the commutativity requirement and thereby broaden our geometric perspective. This leads to the following terminology:

\begin{definition}
    The category of quantum spaces $\mathcal{QS}$ is the opposite category to the category of $\cst$-algebras, i.e., $\mathcal{QS} := \mathcal{C}^*{}^{\textrm{\op}}$.
\end{definition}

In general, the $\cst$-algebra associated with a quantum space $\XX$ will be denoted by $\C(\XX)$. 
In many cases, however, we will denote $\cst$-algebras simply by $\mathcal{M}, \mathcal{N}$, without explicit reference to the corresponding space.

To perform meaningful analysis, we need a Hilbert space structure on our algebra. For the purposes of quantum graphs, it suffices to assume that all $\cst$-algebras are finite-dimensional, thereby placing them within the von Neumann algebraic framework.  
\begin{definition}
Let $\mathcal{M}$ be a finite-dimensional $\cst$-algebra equipped with a faithful and positive functional $\Psi:\mathcal{M}\to \mathbb{C}$. The vector space $\mathcal{M}$, endowed with the scalar product
$$\langle m_1, m_2 \rangle_\Psi := \Psi(m_1^* m_2),$$
is a finite-dimensional Hilbert space, called the GNS space of $(\mathcal{M}, \Psi)$, and denoted by $L^2(\mathcal{M}, \Psi)$.
We will sometimes omit the subscript $\Psi$ in the scalar product when it is clear from the context. Furthermore, we will occasionally identify $\mathcal{M}$ with $L^2(\mathcal{M}, \Psi)$ without explicit mention. If $\mathcal{M}\subset B(H)$ for some Hilbert space, then its commutant will be denoted by  $\mathcal{M}'$. 
\end{definition}

\begin{theorem}[Tomita--Takesaki for finite-dimensional algebras]
    Let $\mathcal{M}$ be a finite-dimensional $\cst$-algebra equipped with a faithful positive functional $\Psi$. There exists a unique one-parameter group of automorphisms $\sigma_z \in \mathrm{Aut}(\mathcal{M})$, defined for all complex parameters $z \in \mathbb{C}$, called the modular group of $\mathcal{M}$ with respect to $\Psi$. This group is uniquely characterized by the KMS (Kubo--Martin--Schwinger) condition:
    $$\Psi(ab) = \Psi(b\sigma_{-i}(a)), \quad \text{for all } a, b \in \mathcal{M}.$$
\end{theorem}

\begin{remark}
    In the finite-dimensional case, any faithful positive functional $\Psi$ on $\mathcal{M}$ can be uniquely represented by a strictly positive, invertible element $\rho \in \mathcal{M}$ (a density matrix) such that:
    $$\Psi(a) = \mathrm{Tr}(\rho a), \quad \text{for all } a \in \mathcal{M}.$$
    This representation follows directly from the Riesz representation theorem applied to the space $L^2(\mathcal{M}, \mathrm{Tr})$. Under this identification, the modular group is given by the explicit formula:
    $$\sigma_z(a) = \rho^{iz} a \rho^{-iz}, \quad \forall a \in \mathcal{M}\; z \in \mathbb{C}.$$

\end{remark}

With the established Hilbert space structure, we can view the algebra's multiplication not merely as an internal operation, but as a linear operator acting on tensor products of Hilbert spaces $L^2(\mathcal{M}, \Psi)$.

\begin{definition}
    Given a $(\mathcal{M}, \Psi)$, where $\mathcal{M}$ is a finite-dimensional $\cst$-algebra and $\Psi$ is a faithful positive functional on $\mathcal{M}$, we define a linear map $m : L^2(\mathcal{M}, \Psi) \otimes L^2(\mathcal{M}, \Psi) \to L^2(\mathcal{M}, \Psi)$ by
    $$m(a\otimes b) := ab, \quad \forall a, b \in \mathcal{M},$$
    and denote by $m^\dagger$ its adjoint, i.e., the unique operator satisfying
    $$\langle v, m(w \otimes u) \rangle_\Psi = \langle m^\dagger(v), w \otimes u \rangle_{\Psi\otimes \Psi} \quad \forall u,v,w \in L^2(\mathcal{M}, \Psi).$$
    Here we use the identification $\mathcal{M} \simeq L^2(\mathcal{M}, \Psi)$.
\end{definition}

A mathematically well-behaved quantum space requires strict compatibility between its algebraic product and its inner product. This geometric feature is captured by the notion of a $\delta$-form.

\begin{definition}
    A positive linear functional $\Psi : \mathcal{M} \to \mathbb{C}$ on a $\cst$-algebra $\mathcal{M}$ is called a $\delta$-form if $m m^\dagger = \delta^2 \mathrm{id}$.
\end{definition}
We remark that in the classical setting it corresponds to a uniform measure on $\C(V)$.
\begin{definition}
    The pair $(\mathcal{M}, \Psi)$ is called a quantum system if $\mathcal{M}$ is a $\cst$-algebra and $\Psi$ is a faithful $\delta$-form. If additionally $\delta=1$, we call the pair $(\mathcal{M}, \Psi)$ a unit quantum system.
\end{definition}

We now have all the necessary ingredients to define the main object of our study. A quantum graph is essentially a quantum system equipped with an adjacency operator that respects both the real structure and the underlying algebraic multiplication.

\begin{definition} \label{q-graph:definition}
    A \emph{quantum graph} is a triple $G = (\mathcal{M}, \Psi, A)$, where:
    \begin{itemize}
        \item $(\mathcal{M}, \Psi)$ is a quantum system with some scaling parameter $\delta$,
        \item $A \in B(L^2(\mathcal{M}, \Psi))$ is an operator satisfying:
        \begin{equation}
            m (A \otimes A) m^\dagger = \delta^2 A, \label{schure:invariance}
        \end{equation}
        \begin{equation}
            A(x)^* = A(x^*) \quad \forall x \in L^2(\mathcal{M},\Psi), \label{A:real}
        \end{equation}
        and
        \begin{equation}
            A^\dagger(x)^*=A^\dagger(x^*) \quad \forall x \in L^2(\mathcal{M},\Psi). \label{A:dagger:real}
        \end{equation}
    \end{itemize}
    The operator $A$ is called the \emph{adjacency operator}. Moreover, if $(\mathcal{M}, \Psi)$ is a unit quantum system, then $G$ is called a unit quantum graph.
\end{definition}

\begin{definition}
    If the operator $A$ satisfies condition \eqref{A:real} (resp. if $A^\dagger$ satisfies condition \eqref{A:dagger:real}), then $A$ (resp. $A^\dagger$) is called $*$-preserving. If $A$ satisfies both \eqref{A:real} and \eqref{A:dagger:real}, it is called a \emph{real operator}.
\end{definition}

\begin{remark}
     We will show in Theorem~\ref{real-commuting:equivalence} that a $*$-preserving operator $A$ is real if and only if it commutes with the modular automorphism group $\sigma_t$. Hence, condition \eqref{A:dagger:real} can be equivalently replaced with the commutation relation $[A, \sigma_t] = 0$ for all $t \in \mathbb{R}$.
\end{remark}

Just like for classical graphs, we can impose algebraic conditions on the adjacency operator to capture geometric properties in the quantum realm.

\begin{definition}
    Define the map $\eta : \mathbb{C} \ni \lambda \mapsto \lambda \I \in \mathcal{M}$. A graph $G = (\mathcal{M}, \Psi, A)$ is said to be \emph{undirected} if it satisfies
    \begin{equation}
        (\mathrm{id} \otimes \eta^\dagger m)(\mathrm{id} \otimes A \otimes \mathrm{id})(m^\dagger \eta \otimes \mathrm{id}) = A. \label{transposition:invariance}
    \end{equation}
    Furthermore, $G$ is called \emph{reflexive} (respectively, \emph{irreflexive}) if
    \begin{equation}
        \delta^{-2}m(A \otimes \mathrm{id})m^\dagger = \mathrm{id}, \quad \text{(or } \delta^{-2}m(A \otimes \mathrm{id})m^\dagger = 0 \text{, respectively)}. \label{loops:noloops}
    \end{equation}
\end{definition}

\begin{remark}
    Some authors define an undirected graph by requiring $A$ to be self-adjoint. It is well known that these definitions are equivalent -- we will discuss this equivalence in more detail in Theorem~\ref{undirected_equivalence}. 
\end{remark}

The adjacency operator naturally induces a specific subspace of operators that corresponds to the ``edges'' of the underlying quantum graph.

\begin{definition}\label{projPA}
    The spaces
    $$
    \begin{aligned}
&S_L(G)=\operatorname{span}\left\{\delta^{-2}m(A\otimes X)m^\dagger \mid X\in B(L^2(\mathcal{M}, \Psi))\right\}\subset B(L^2(\mathcal{M}, \Psi))\\
&S_R(G)=\operatorname{span}\left\{\delta^{-2}m(X\otimes A)m^\dagger \mid X\in B(L^2(\mathcal{M}, \Psi))\right\}\subset B(L^2(\mathcal{M}, \Psi))
    \end{aligned}$$ are called the left and right quantum space of edges, respectively. 
\end{definition}

One readily verifies that if $A$ satisfies \eqref{schure:invariance}, then the maps
\[
\begin{aligned}
&P_A: B(\mathcal{M}) \ni X \mapsto \delta^{-2}m(A\otimes X)m^\dagger \in S_L(G)\\
&P_A': B(\mathcal{M}) \ni X \mapsto \delta^{-2}m(X\otimes A)m^\dagger \in S_R(G)
\end{aligned}
\]
are projections. However, they are not orthogonal with respect to the Hilbert--Schmidt inner product (see Theorem~\ref{not_orthogonal}). We observe that $P_A$ and $P_A'$ are orthogonal with respect to different inner products, introduced in Definition~\ref{R_L_scalar_product} and characterized in Theorem~\ref{orthogonal_P_A}. 

It is instructive to verify how standard classical graphs fit in this framework as a special case.

\begin{example}
    Let $G = (\C([n]), \Psi, A)$ be a classical graph, where:
    \begin{itemize}
        \item $\Psi(x) = \frac{1}{n}\sum_{i \in [n]} x_i$,
        \item $A \in \operatorname{Mat}_n(\mathbb{C})$ is an adjacency matrix.
    \end{itemize}
 A straightforward calculation  yields $m^\dagger(e_i) = n\, e_i \otimes e_i$, where $e_i$ denotes the canonical basis of $\C([n])$ and shows $\Psi$ is a $\sqrt{n}$-form. Furthermore, condition \eqref{schure:invariance} is equivalent to $A_{ij} \in \{0,1\}$ for all $i, j \in [n]$; condition \eqref{transposition:invariance} ensures symmetry $A = A^T$; and condition \eqref{loops:noloops} specifies that $A$ has ones on the diagonal (or zeros in the irreflexive case).
\end{example}

\begin{example}\label{complete-q-graph}
    For any quantum system $(\mathcal{M}, \Psi)$, the map $A_c:=\delta^2\Psi(\cdot)\I$ yields a quantum graph $(\mathcal{M},\Psi, A_c)$. Such a graph is called a complete quantum graph on $(\mathcal{M}, \Psi)$. Moreover, for any quantum system, the identity operator $A_l= \mathrm{id}$ is always an adjacency operator.
\end{example}

\section{The Schur product properties}
\label{sec:Schur}

In this chapter, we introduce fundamental properties and identities related to the Schur product and the operator $m^\dagger$. Many of these results have already been established in the more general setting in \cite{Daws_2024}. In particular, Theorem~\ref{schur_product} and Proposition~\ref{conjugate and transpose} are direct consequences of \cite[Proposition~5.3]{Daws_2024}.

To manipulate operators on the GNS space effectively, we frequently rely on basis expansions. The following technical lemma is a good example of such. 

\begin{lemma}
    For any orthonormal basis $(\xi_j)_{j\in J}$ of $L^2(\mathcal{M}, \Psi)$, the following identity holds:
    $$\sum_j \ket{\xi_j^*}\bra{\sigma_{-i}(\xi_j^*)}=\mathrm{id}=\sum_j \ket{\sigma_{-i}(\xi_j^*)}\bra{\xi_j^*}.$$
\end{lemma}

\begin{proof}
    For any $a\in L^2(\mathcal{M}, \Psi)$, we compute:
    \begin{align*}
        \sum_j \ket{\xi_j^*}\bra{\sigma_{-i}(\xi_j^*)}a
        &=\sum_j \langle a^*, \xi_j \rangle \xi_j^*
        =\left(\sum_{j} \langle\xi_j, a^* \rangle \xi_j\right)^*
        =(a^*)^*=a.
    \end{align*}
    The second identity follows from the self-adjointness of the operator $\mathrm{id}$.
\end{proof}

With this identity in hand, we can explicitly determine the action of $m^\dagger$. This result allows us to compute the co-multiplication explicitly.

\begin{lemma}\label{m-dagger-identities}
    Let $(\mathcal{M}, \Psi)$ be a quantum system with arbitrary $\delta$, and let $(\xi_j)_{j\in J}$ be any orthonormal basis of $L^2(\mathcal{M}, \Psi)$. Then the operator $m^\dagger:\mathcal{M}\to \mathcal{M}\otimes \mathcal{M}$ is given by the following formulas:
    $$m^\dagger(a)=\sum_j a \sigma_{-i}(\xi_j^*)\otimes \xi_j=\sum_j \xi_j \otimes \xi_j^*a \quad \forall a\in \mathcal{M}.$$
\end{lemma}

\begin{proof}
    Let $a,b,c\in\mathcal{M}$. Then:
    \begin{align*}
        \langle b\otimes c, m^\dagger a \rangle_{\Psi \otimes \Psi}
        &=\langle bc, a \rangle_\Psi\\
        &=\left\langle b\sum_j \ket{\xi_j}\bra{\xi_j} c, a\right\rangle_\Psi\\
        &=\sum_j\langle c, \xi_j \rangle_\Psi \langle b\xi_j, a\rangle_\Psi\\
        &=\sum_{j}\langle b, a \sigma_{-i}(\xi_j^*)\rangle_\Psi \langle c, \xi_j \rangle_\Psi\\
        &=\left\langle b \otimes c, \sum_j a \sigma_{-i}(\xi_j^*) \otimes \xi_j\right\rangle_{\Psi \otimes \Psi},
    \end{align*}
    and similarly for the second identity:
    \begin{align*}
        \langle b\otimes c, m^\dagger a \rangle_{\Psi \otimes \Psi}
        &=\langle bc, a \rangle_\Psi\\
        &=\left\langle b, a\sigma_{-i}(c^*)\right\rangle_\Psi\\
        &=\left\langle b,\sum_j \ket{\xi_j}\bra{\xi_j}a\sigma_{-i}(c^*) \right\rangle_\Psi\\
        &=\sum_j \langle b, \xi_j \rangle_\Psi \langle\xi_j, a \sigma_{-i}(c^*)\rangle_\Psi\\
        &=\sum_j \langle b, \xi_j \rangle_\Psi \langle\xi_j c , a\rangle_\Psi\\
        &=\sum_j \langle b, \xi_j \rangle_\Psi \langle c, \xi_j^* a \rangle_\Psi\\
        &=\left \langle b\otimes c, \sum_j \xi_j \otimes \xi_j^* a\right \rangle_{\Psi \otimes \Psi}.
    \end{align*}
\end{proof}

A direct application of this explicit form reveals a fundamental compatibility between the modular group, the basis choice, and the spatial parameter $\delta$.

\begin{corollary} \label{xi_j_xi_j*}
    Let $(\mathcal{M}, \Psi)$ be a quantum system with some $\delta$. Then:
    $$\sum_j \sigma_{-i}(\xi_j^*)\xi_j=\sum_j \xi_j \xi_j^*=\delta^2 \I.$$
\end{corollary}

\begin{proof}
    The statement follows from Lemma \ref{m-dagger-identities} and the identity $m m^\dagger=\delta^2 \mathrm{id}$.
\end{proof}

These algebraic tools immediately yield structural results for quantum graphs. For instance, just as we can construct the complement of a classical graph by swapping edges and non-edges, we can define the complement of a quantum graph algebraically.

\begin{lemma}\label{graph_complement}
    Let $(\mathcal{M}, \Psi, A)$ be a quantum graph with some $\delta$ and $(\mathcal{M}, \Psi, A_c)$ be the complete quantum graph from example \ref{complete-q-graph} . Then $(\mathcal{M}, \Psi, \widetilde{A}=A_c-A)$ is another quantum graph.
\end{lemma}
\begin{proof}
    It is straightforward that $\delta^{-2}A_c$ is the identity of the Schur product and hence: 
    \begin{align*}
        \widetilde{A}\cdot_S \widetilde{A}&=(A_c-A)\cdot_S (A_c-A)\\
        &=A_c \cdot_S A_c-A_c \cdot_S A-A\cdot_S A_c+A\cdot_S A\\
        &=\delta^2A_c-2\delta^2A+\delta^2A=\delta^2 \widetilde{A}.
    \end{align*}
\end{proof}

 By conjugating the standard tensor product of rank-one operators with $m$ and $m^\dagger$, we obtain an operation that perfectly mimics the classical entry-wise matrix product, extended to the non-commutative setting.

\begin{theorem}\label{schur_product}
    Let $a_j, b_j, c_s, d_s \in \mathcal{M}$ for $j \in I$, $s \in S$. Then:
    $$m\left( \sum_{j} \ket{a_j}\bra{b_j} \otimes \sum_{s} \ket{c_s}\bra{d_s} \right) m^\dagger = \sum_{j,s} \ket{a_j c_s}\bra{b_j d_s}.$$
\end{theorem}

\begin{proof}
    Let $a, b, c, d \in \mathcal{M}$. Then:
    \begin{align*}
        m\left( \ket{a}\bra{b} \otimes \ket{c}\bra{d} \right) m^\dagger e
        &= m(a \otimes c) \langle b \otimes d, m^\dagger e \rangle \\
        &= ac \langle m(b \otimes d), e \rangle \\
        &= \ket{ac} \bra{bd} e,
    \end{align*}
    and the result follows by linearity.
\end{proof}

\begin{definition}
    The operation $M_1\cdot_S M_2:=m(M_1\otimes M_2)m^\dagger$ will be referred to as the Schur product.
\end{definition}

Naturally, this product inherits key properties from the underlying algebraic multiplication and the adjoint operation.

\begin{corollary}
    The Schur product is associative, and $(M_1\cdot_S M_2)^\dagger=M_1^\dagger\cdot_S M_2^\dagger$.
\end{corollary}
Let us now introduce the notation related to the opposite algebra and the modular theory for them. 

\begin{definition}
The opposite algebra $\mathcal{M}^{\op}$ of a von Neumann algebra $\mathcal{M}$ has the same underlying vector space and involution, but the reversed product:
$$a \cdot_{\op} b := b a.$$
For any faithful $\delta$-form on $\mathcal{M}$, the linear functional $\Psi^\op$ on $\mathcal{M}^\op$ defined by $\Psi^\op(x^\op)=\Psi(x)$  is  a faithful $\delta$-form on $\mathcal{M}^{\op}$. In particular, this implies $\Psi^{\op}(a \cdot_{\op} b) = \Psi(ba)$. We denote the GNS Hilbert space associated with the pair $(\mathcal{M}^{\op}, \Psi^{\op})$ by $L^2(\mathcal{M}^{\op}, \Psi^{\op})$, and we implicitly identify $\mathcal{M}^{\op}$ with $L^2(\mathcal{M}^{\op}, \Psi^{\op})$.

\end{definition}

\begin{lemma}
    Let us consider the map $\iota : \mathcal{M} \to \mathcal{M}^{\op}$ given by $\iota(a)=\sigma_{i/2}(a)$. The map $\iota$ viewed as operator $L^2(\mathcal{M},\Psi)\to L^2(\mathcal{M}^\op,\Psi^\op)$ is unitary.
\end{lemma}

\begin{proof}
    We compute:
    \begin{align*}
        \langle\iota(a), \iota(b)\rangle_{\Psi^\op}
        &=\Psi^\op(\sigma_{-i/2}(a^*)\cdot_\op \sigma_{i/2}(b))
        =\Psi(\sigma_{i/2}(b)\sigma_{-i/2}(a^*))\\
        &=\Psi(b\sigma_{-i}(a^*))
        =\Psi(a^*b)
        =\langle a, b \rangle_\Psi.
    \end{align*}
\end{proof}

This unitary equivalence allows us to faithfully represent both the algebra and its opposite algebra on the same Hilbert space.

\begin{definition}
Let $(\mathcal{M}, \Psi)$ be a quantum system with some $\delta$. We define two $*$-homomorphisms
$$\pi_\Psi : \mathcal{M} \to B(L^2(\mathcal{M}, \Psi)), \qquad \pi_\Psi^{\op} : \mathcal{M}^{\op} \to B(L^2(\mathcal{M}, \Psi))$$
by
$$\pi_\Psi(a)b = a b, \qquad \pi_\Psi^{\op}(a)b = b\sigma_{-i/2}(a).$$
\end{definition}
Let us note that the maps $\pi$ and $\pi^\op$ are indeed $*$-homomorphisms:
    $$\langle a,\pi^\op(b)c\rangle_\Psi=\Psi(a^*c\sigma_{-i/2}(b))=\Psi(\sigma_{i/2}(b)a^*c)=\langle a\sigma_{-i/2}(b^*), c\rangle_\Psi=\langle \pi^\op(b^*)a, c \rangle_\Psi.$$
We conclude this section by tying together the Schur product and these representations. The following proposition provides powerful computational identities that allow us to translate tensor operations directly into left and right multiplication operators.

\begin{proposition} \label{LR_representation}
    Let $(\mathcal{M}, \Psi)$ be a quantum system. For any $B = \sum_s \ket{b_s} \bra{a_s}\in B(L^2(\mathcal{M}, \Psi))$, we have:
    \begin{itemize}
        \item[i)] $m(B \otimes \mathrm{id}) m^\dagger = \sum_s \pi_\Psi(b_s a_s^*)$,
        \item[ii)] $m(\mathrm{id} \otimes B) m^\dagger = \sum_s \pi_\Psi^\op(\sigma_{-i/2}(a_s^*)\sigma_{i/2}(b_s))$.
    \end{itemize}
\end{proposition}

\begin{proof}
    Let $(\xi_j)_{j \in I}$ be an orthonormal basis of $L^2(\mathcal{M}, \Psi)$. Then:
    \begin{align*}
        m\left( \sum_s \ket{b_s} \bra{a_s} \otimes \mathrm{id} \right) m^\dagger
        &= \sum_{s,j} \pi_\Psi(b_s) \ket{\xi_j} \bra{\xi_j} \pi_\Psi(a_s^*) 
        = \sum_s \pi_\Psi(b_s a_s^*),
    \end{align*}
    and for the right action:
    \begin{align*}
        m\left( \sum_s \mathrm{id}\otimes \ket{b_s} \bra{a_s} \right) m^\dagger
        &=\sum_{s,j}\ket{\xi_j b_s}\bra{\xi_j a_s}\\
        &=\sum_{j,s}\pi_\Psi^\op(\sigma_{i/2}(b_s))\ket{\xi_j}\bra{\xi_j}\pi_\Psi^\op(\sigma_{-i/2}(a_s^*))\\
        &=\sum_s\pi_\Psi^\op(\sigma_{-i/2}(a_s^*)\sigma_{i/2}(b_s)).
    \end{align*}
\end{proof}

\begin{remark} \label{cdotpi}
    Analogously, for $a,b \in \mathcal{M}$, $B \in L^2(\mathcal{M}, \Psi)$, one has:
    \begin{itemize}
        \item[i)] $m\left( \ket{b}\bra{a} \otimes B \right) m^\dagger = \pi_\Psi(b) B \pi_\Psi(a^*)$,
        \item[ii)] $m\left( B \otimes \ket{b}\bra{a} \right) m^\dagger = \pi_\Psi^{\op}(\sigma_{i/2}(b)) B \pi_\Psi^{\op}(\sigma_{-i/2}(a^*))$.
    \end{itemize}
\end{remark}

\section{Left and right transpositions}
\label{sec:transpositions}

In this section, we introduce the notion of quantum transpositions in the framework of quantum graphs. To define quantum incidence operators and characterize their adjoints coherently, the classical transposition must be replaced by two dual algebraic operations, namely the \emph{left} and \emph{right transpositions}.

\begin{proposition} \label{conjugate and transpose}
Let $(\mathcal{M}, \Psi)$ be a quantum system. Then for all families $(a_s), (b_s) \subseteq \mathcal{M}$ indexed by $s \in S$, the following identities hold:
\begin{itemize}
    \item[i)] $(\mathrm{id} \otimes \eta^\dagger m)(\mathrm{id} \otimes \sum_{s\in S} \ket{a_s}\bra{b_s} \otimes \mathrm{id})(m^\dagger \eta \otimes \mathrm{id}) = \sum_{s \in S} \ket{\sigma_{-i}(b_s^*)}\bra{a_s^*}$,
    \item[ii)] $(\eta^\dagger m \otimes \mathrm{id})(\mathrm{id} \otimes \sum_{s\in S} \ket{a_s}\bra{b_s} \otimes \mathrm{id})(\mathrm{id} \otimes m^\dagger \eta) = \sum_{s \in S} \ket{b_s^*}\bra{\sigma_{-i}(a_s^*)}$.
\end{itemize}
\end{proposition}

\begin{proof}
i) Let $(\xi_j)_{j \in J}$ be an orthonormal basis of $L^2(\mathcal{M}, \Psi)$. Recall that $m^\dagger(\I) = \sum_{j\in J} \sigma_{-i}(\xi_j^*) \otimes \xi_j$. Then:
\begin{align*}
\langle c, (\mathrm{id} \otimes \eta^\dagger m)(\mathrm{id} \otimes \ket{a}\bra{b} \otimes \mathrm{id})(m^\dagger \eta \otimes \mathrm{id})d \rangle
&= \sum_{j \in J} \langle c, (\mathrm{id} \otimes \eta^\dagger m)(\sigma_{-i}(\xi_j^*) \otimes \langle b, \xi_j \rangle a \otimes d) \rangle \\
&= \sum_{j \in J} \langle a^*, d \rangle \langle b, \xi_j \rangle \langle c, \sigma_{-i}(\xi_j^*) \rangle \\
&= \sum_{j \in J} \langle a^*, d \rangle \langle b, \xi_j \rangle \langle \xi_j, c^* \rangle \\
&= \langle a^*, d \rangle \langle b, c^* \rangle \\
&= \langle c, \ket{\sigma_{-i}(b^*)}\bra{a^*} d \rangle.
\end{align*}

ii) Similarly,
\begin{align*}
(\eta^\dagger m \otimes \mathrm{id})(\mathrm{id} \otimes \ket{a}\bra{b} \otimes \mathrm{id})(\mathrm{id} \otimes m^\dagger \eta)c
&= \sum_{j \in J} (\eta^\dagger m \otimes \mathrm{id})(c \otimes \langle b, \sigma_{-i}(\xi_j^*) \rangle a \otimes \xi_j) \\
&= \sum_{j \in J} \langle b, \sigma_{-i}(\xi_j^*) \rangle \langle c^*, a \rangle \xi_j \\
&= \sum_{j \in J} \langle \sigma_{-i}(a^*), c \rangle \langle \xi_j, b^* \rangle \xi_j \\
&= \ket{b^*} \bra{\sigma_{-i}(a^*)} c.
\end{align*}
Linearity completes the proof.
\end{proof}

The preceding computational identities naturally extend to the entire space of bounded operators, leading to the following definitions of the left and right transposition operators on $B(L^2(\mathcal{M}, \Psi))$.

\begin{definition}
The linear maps from Proposition~\ref{conjugate and transpose} will be referred to as the right and left transpositions, denoted respectively by:
$$M^{TR} := (\mathrm{id} \otimes \eta^\dagger m)(\mathrm{id} \otimes M \otimes \mathrm{id})(m^\dagger \eta \otimes \mathrm{id}),$$
$$M^{TL} := (\eta^\dagger m \otimes \mathrm{id})(\mathrm{id} \otimes M \otimes \mathrm{id})(\mathrm{id} \otimes m^\dagger \eta).$$
\end{definition}

Let us first observe that $TL$ and $TR$ are mutual inverses.

\begin{lemma}
\label{lemma:trans}
For any $M \in B(L^2(\mathcal{M}, \Psi))$, we have:
$$(M^{TL})^{TR} = (M^{TR})^{TL} = M.$$
\end{lemma}

\begin{proof}
Let $M = \ket{a}\bra{b}$. Then:
$$(M^{TL})^{TR} = (\ket{b^*} \bra{\sigma_{-i}(a^*)})^{TR} = \ket{\sigma_{-i}(\sigma_i(a))}\bra{b} = \ket{a}\bra{b} = M.$$
The proof of $(M^{TR})^{TL} = M$ is analogous.
\end{proof}

This structural duality immediately implies that left-invariance and right-invariance are completely equivalent notions.

\begin{remark}
An immediate consequence of Lemma~\ref{lemma:trans} is that left symmetry ($M = M^{TL}$) and right symmetry ($M = M^{TR}$) are equivalent. Indeed:
$$M = M^{TL} \Rightarrow M = (M^{TL})^{TR} = M^{TR},$$
$$M = M^{TR} \Rightarrow M = (M^{TR})^{TL} = M^{TL}.$$
\end{remark}

Beyond their behavior under composition, these transpositions interact nicely with algebraic products, traces, and conjugations, closely mirroring the classical properties of the standard matrix transpose.

\begin{proposition} \label{Transpose_features}
Let $(\mathcal{M}, \Psi)$ be a quantum system and $M, M_1, M_2 \in B(L^2(\mathcal{M}, \Psi))$. Then:
\begin{enumerate}
    \item[i)] $(M_1 \cdot_S M_2)^{TR} = M_2^{TR} \cdot_S M_1^{TR}$, \quad $(M_1 \cdot_S M_2)^{TL} = M_2^{TL} \cdot_S M_1^{TL}$,
    \item[ii)] $(M_1 M_2)^{TR} = M_2^{TR} M_1^{TR}$, \quad $(M_1 M_2)^{TL} = M_2^{TL} M_1^{TL}$,
    \item[iii)] $\mathrm{Tr}(M) = \mathrm{Tr}(M^{TR}) = \mathrm{Tr}(M^{TL})$,
    \item[iv)] $(M^{TL})^\dagger = (M^\dagger)^{TR}$, \quad $(M^{TR})^\dagger = (M^\dagger)^{TL}$.
\end{enumerate}
\end{proposition}

\begin{proof}
We prove the statements for left transposition; the right case is analogous.

i) Let $M_j = \ket{a_j}\bra{b_j}$. Then:
$$(M_1 \cdot_S M_2)^{TL} = \left(\ket{a_1 a_2}\bra{b_1 b_2}\right)^{TL} = \ket{b_2^* b_1^*}\bra{\sigma_{-i}(a_2^*) \sigma_{-i}(a_1^*)} = M_2^{TL} \cdot_S M_1^{TL}.$$

ii)
\begin{align*}
(M_1 M_2)^{TL}
&= \langle b_1, a_2 \rangle \cdot \left(\ket{a_1}\bra{b_2}\right)^{TL}
= \langle \sigma_{-i}(a_2^*), b_1^* \rangle \cdot \ket{b_2^*}\bra{\sigma_{-i}(a_1^*)} \\
&= M_2^{TL} M_1^{TL}.
\end{align*}

iii) Using the cyclicity of the trace and assuming $M=\ket{a}\bra{b}$:
$$\mathrm{Tr}(M) = \mathrm{Tr}(\ket{a}\bra{b}) = \Psi(b^* a),$$
$$\mathrm{Tr}(M^{TL}) = \mathrm{Tr}(\ket{b^*}\bra{\sigma_{-i}(a^*)}) = \Psi(\sigma_i(a) b^*) = \Psi(b^* a).$$

iv)
$$((\ket{a}\bra{b})^{TL})^\dagger = (\ket{b^*}\bra{\sigma_{-i}(a^*)})^\dagger = \ket{\sigma_{-i}(a^*)} \bra{b^*} = (\ket{b}\bra{a})^{TR}.$$
\end{proof}

Given an operator $M \in B(L^2(\mathcal{M},\Psi))$, one can consider the following three symmetry properties: self-adjointness, invariance under the left transposition (equivalently, the right transposition), and commutation with the involution $*$ on $\mathcal{M}$. These properties are independent in general. Nevertheless, any two of them imply the third.
\begin{theorem}\label{undirected_equivalence}
    Let $(\mathcal{M}, \Psi)$ be a quantum system. Then for any $M \in B(L^2(\mathcal{M}, \Psi))$, any two of the following conditions imply the third:
    \begin{itemize}
        \item[i)] $M$ is right (and left) symmetric: $M=M^{TR}$ (and $M=M^{TL}$),
        \item[iI)] $M$ is $*$-preserving: $(Ma)^*=M(a^*)$ for all $a\in L^2(\mathcal{M}, \Psi)$,
        \item[iii)] $M$ is self-adjoint: $M=M^\dagger$.
    \end{itemize}
\end{theorem}

\begin{proof}
    Let us denote $\overline{M}(a):=M(a^*)^*$. Note that $M$ is $\ast$-preserving if and only if $\overline{M}=M$. Without loss of generality, assume that $M=\ket{v}\bra{w}$. Then:
    \begin{align*}
        \overline{M}(a)&:=(Ma^*)^*= \left(\langle w, a^* \rangle\, v\right)^*=\langle a^*, w \rangle v^*\\
        &=\langle \sigma_{-i}(w^*), a \rangle v^*
        =\ket{v^*}\bra{\sigma_{-i}(w^*)}a.
    \end{align*}
    It follows that $\overline{M}=(M^\dagger)^{TL}=(M^{TR})^\dagger$, hence $M^{TR}=\overline{M}^\dagger$ and $M^\dagger=\overline{M}^{TR}$, which completes the proof.
\end{proof}

As an immediate consequence of this triality, we obtain the following characterization of how the property of being real behaves under taking adjoints.

\begin{lemma}
    If $M$ is real, then $M^\dagger=M^{TR}=M^{TL}$.
\end{lemma}
\begin{proof}
    First, recall that $M$ is $*$-preserving, which implies $M=\overline{M}=(M^\dagger)^{TL}$, and hence $M^{TR}=M^\dagger$. Since $M^\dagger$ is also $*$-preserving, we immediately obtain $M^{\dagger}=\overline{M^\dagger}=M^{TL}$.
\end{proof}

\begin{example}
    If $M$ is only $*$-preserving but not real, then the equality $M^\dagger=M^{TL}$ does not hold in general. To see this, consider the quantum system $(\Mat_2, \Tr(\rho \cdot))$, where 
    \[\rho =\left[\begin{array}{cc}
        p & 0 \\
        0 & q \\
    \end{array}\right]\]
    for some non-zero $p,q \in \RR$ such that $p+q=1$. Moreover, let $A=\ket{\I}\bra{w}$, where
    \[w=\left[\begin{array}{cc}
        0 & 1 \\
        \gamma^{-1} & 0 \\
    \end{array}\right]\]
    for $\gamma=\frac{p}{q}$.
    Then $A^\dagger=\ket{w}\bra{\I}$ and $A^{TR}=\ket{\sigma_{-i}(w^*)}\bra{\I}=A^\dagger$, because
    \[\sigma_{-i}(w^*)=\left[\begin{array}{cc}
        p & 0 \\
        0 & q
    \end{array}\right]\left[\begin{array}{cc}
        0 & \gamma^{-1} \\
        1 & 0
    \end{array}\right]\left[\begin{array}{cc}
        \frac{1}{p} & 0 \\
        0 & \frac{1}{q} \\
    \end{array}\right]=\left[\begin{array}{cc}
        0 & \frac{\gamma^{-1}p}{q} \\
        \frac{q}{p} & 0 \\
    \end{array}\right]=w.\]
    On the other hand, we have $A^{TL}=\ket{w^*}\bra{\I}\neq A^\dagger$.
\end{example}

The following corollary is an immediate consequence of the preceding theorem and lemma.

\begin{corollary}
    For any quantum graph $(\mathcal{M}, \Psi, A)$, we have $A^{\dagger}=A^{TL}=A^{TR}$. 
\end{corollary}

Finally, we turn our attention to the interplay between real operators and the modular automorphism group.

\begin{lemma} \label{sigma_commuting}
    Let $M\in B(L^2(\mathcal{M}, \Psi))$ be real. Then $\sigma_{t}\circ M=M \circ \sigma_{t}$ for all $t \in \RR$.
\end{lemma}

\begin{proof}
    Without loss of generality, we assume that $M=\ket{a}\bra{b}$ and that $M$ is real. By the previous lemma, $M^\dagger=M^{TL}=M^{TR}$. This allows us to compute:
    \begin{align*}
        M\circ \sigma_{-i}&=\overline{M}\circ \sigma_{-i}\\
        &=(M^{TR})^\dagger \circ\sigma_{-i}\\
        &=(M^{TL})^\dagger \circ\sigma_{-i}\\
        &=\ket{\sigma_{-i}(a^*)}\bra{b^*}\circ \sigma_{-i}\tag{$\star$}\\
        &=\sigma_{-i}\circ \ket{a^*}\bra{\sigma_{-i}(b^*)}\\
        &=\sigma_{-i}\circ (M^\dagger)^{TL}\\
        &=\sigma_{-i}\circ \overline{M}\\
        &=\sigma_{-i}\circ M,
    \end{align*}
    where we used Theorem~\ref{undirected_equivalence} in the penultimate step, and step $(\star)$ arises from the KMS condition. Hence, $\sigma_{-i}\circ M=M\circ \sigma_{-i}$, which implies $M\circ \sigma_t=\sigma_t\circ M$ (see \cite{Wasilewski_2024}, Lemma 2.1).
\end{proof}

\begin{corollary}
    For any quantum graph $(\mathcal{M}, \Psi, A)$, we have $\sigma_t \circ A=A\circ \sigma_t$ for all $t\in \CC$, and the same holds for $A^{\dagger}$.
\end{corollary}

In particular, the following theorem holds.

\begin{theorem} \label{real-commuting:equivalence}
    Any $*$-preserving $M\in B(L^2(\mathcal{M},\Psi))$ is real if and only if $[M, \sigma_t]=0$ for all $t \in \RR$.
\end{theorem}
\begin{proof}
    The forward direction has already been established in Lemma~\ref{sigma_commuting}. Conversely, suppose that $[M, \sigma_t]=0$. Commutativity with the modular automorphism group implies that $M^{TL}=M^{TR}$ (which is a direct consequence of Proposition~\ref{conjugate and transpose}). Since $M$ is $*$-preserving, we have $M^\dagger=M^{TR}$. Furthermore, because $\overline{M^\dagger} = M^{TL}$, we immediately obtain $\overline{M^\dagger} = M^{TL} = M^{TR} = M^\dagger$, meaning that $M^\dagger$ is also $*$-preserving. Thus, $M$ is real.
\end{proof}
\section{Edge spaces and their connection with bimodules}
\label{sec:edge}

This section is devoted to the observation that, in general, the projection $P_A$ defined in Definition~\ref{projPA} is not orthogonal with respect to the standard Hilbert--Schmidt inner product. Nevertheless, it becomes self-adjoint when considered with respect to a suitably adapted inner product. We then use these properties to investigate the structure of the edge spaces $S(G)$ of a quantum graph. Finally, we establish that the edge spaces $S(G)\subseteq B(L^2(\mathcal{M},\Psi))$ are closely related to the closed $\mathcal{M}'$--$\mathcal{M}'$ bimodules appearing in Weaver's definition of a quantum graph~\cite{NW2015}. Moreover, we show that the entire construction admits an equivalent formulation in terms of $\mathcal{M}$--$\mathcal{M}$ bimodules.

To establish a rigorous connection between our straightforward approach and the structural bimodule framework, we first analyze the behavior of rank-one operators under transpositions and the modular automorphism group. We begin with a technical lemma.

\begin{lemma}
    Let $(\mathcal{M}, \Psi)$ be a finite-dimensional quantum system. For any $a,b\in L^2(\mathcal{M}, \Psi)$, we have:
    \begin{itemize}
        \item[i)] $$\ket{a^*}\bra{b^*}=(\ket{a}\bra{b}^\dagger)^{TL}\circ \sigma_i,$$
        \item[ii)] $$\ket{\sigma_{-i}(a^*)}\bra{\sigma_{-i}(b^*)}=\sigma_{-i}\circ (\ket{a}\bra{b}^\dagger)^{TL}.$$
    \end{itemize}
\end{lemma}

\begin{proof}
    Observe first that (i) and (ii) are equivalent due to the substitution $M \mapsto \sigma_{-i} \circ M \circ \sigma_{-i}$. Hence it suffices to prove (i):
    \begin{align*}
        (\ket{a}\bra{b}^\dagger)^{TL}\circ \sigma_i
        &= \ket{b}\bra{a}^{TL} \circ \sigma_i \\
        &= \ket{a^*}\bra{\sigma_{-i}(b^*)}\circ \sigma_i \\
        &= \ket{a^*}\bra{b^*}.
    \end{align*}
\end{proof}

This operator identity allows us to track how the Schur product behaves with respect to the standard Hilbert--Schmidt inner product. The next lemma provides the exact dual relations needed for this calculation.

\begin{lemma}\label{schur_conjugate}
    Let $(\mathcal{M}, \Psi)$ be a finite-dimensional quantum system and $M, N, Z\in B(L^2(\mathcal{M}, \Psi))$. Then
    $$\langle Z, M\cdot_S N\rangle_{HS}=\langle ((M^\dagger)^{TL}\circ \sigma_i) \cdot_S Z, N\rangle_{HS},$$
    $$\langle Z, N\cdot_S M\rangle_{HS}=\langle Z\cdot_S (\sigma_{-i}\circ (M^{\dagger})^{TL}), N\rangle_{HS}.$$
\end{lemma}

\begin{proof}
    Without loss of generality, assume that $M=\ket{a}\bra{b}$. Then:
    \begin{align*}
        \langle Z, \ket{a}\bra{b}\cdot_S N\rangle_{HS}
        &=\langle Z, \pi_\Psi(a)N\pi_\Psi(b^*)\rangle_{HS}\\
        &=\mathrm{Tr}(Z^\dagger\pi_\Psi(a)N\pi_\Psi(b^*))\\
        &=\mathrm{Tr}(\pi_\Psi(b^*)Z^\dagger \pi_\Psi(a)N)\\
        &=\mathrm{Tr}(\ket{b^*}\bra{a^*}\cdot_S Z^\dagger N)\\
        &=\mathrm{Tr}((\ket{a^*}\bra{b^*}\cdot_S Z)^\dagger N)\\
        &=\langle \ket{a^*}\bra{b^*}\cdot_S Z, N\rangle_{HS}\\
        &=\langle ((\ket{a}\bra{b}^\dagger)^{TL}\circ \sigma_i)\cdot_S Z, N\rangle_{HS}.
    \end{align*}
    Similarly,
    \begin{align*}
        \langle Z, N\cdot_S M\rangle_{HS}
        &=\mathrm{Tr}(Z^\dagger N\cdot_S \ket{a}\bra{b})\\
        &=\mathrm{Tr}(Z^\dagger \cdot_S \ket{\sigma_{-i}(b^*)}\bra{\sigma_{-i}(a^*)} N)\\
        &=\langle Z\cdot_S(\sigma_{-i}\circ \ket{a^*}\bra{\sigma_{-i}(b^*)}), N\rangle_{HS}\\
        &=\langle Z\cdot_S(\sigma_{-i}\circ (\ket{a}\bra{b}^\dagger)^{TL}), N\rangle_{HS}.
    \end{align*}
\end{proof}

Having established these identities, we are now in a position to determine the Hilbert--Schmidt adjoint of $P_A$ and of its opposite counterpart $P_A'$, introduced below.

\begin{theorem} \label{not_orthogonal}
Let $A$ be the adjacency operator of a quantum graph $(\mathcal{M}, \Psi, A)$ and let $P_A, P_A':B(L^2(\mathcal{M}, \Psi))\to B(L^2(\mathcal{M}, \Psi))$ be defined by:
$$P_A(X)=\delta^{-2}A\cdot_S X, \qquad P_A'(X)=\delta^{-2}X\cdot_S A.$$
Then
$$P_A^\dagger=P_{A\circ \sigma_i}, \qquad (P_A')^\dagger=P_{\sigma_{-i}\circ A}',$$
where the adjoint is taken with respect to the Hilbert--Schmidt scalar product.
\end{theorem}

\begin{proof}
    The statement follows directly from Lemma~\ref{schur_conjugate}, Theorem \ref{undirected_equivalence}, and the fact that $A$ is real.
\end{proof}

It turns out that a modular modification of the Hilbert--Schmidt scalar product gives rise to selfadjointness of $P_A$.   

\begin{definition} \label{R_L_scalar_product}
    We define the \emph{left} and \emph{right} scalar products on $B(L^2(\mathcal{M}, \Psi))$ by:
    $$\langle X,Y\rangle_L=\mathrm{Tr}((\sigma_i \circ X)^\dagger Y), \quad \langle X, Y\rangle_R=\mathrm{Tr}((X \circ \sigma_{-i})^\dagger Y),$$
    for all $X,Y\in B(L^2(\mathcal{M}, \Psi))$.
    The corresponding adjoints will be denoted by ${}^{\dagger_L}$ and ${}^{\dagger_R}$, respectively.
\end{definition}

\begin{remark}
To verify that these expressions define genuine inner products, it suffices to use the identity $\sigma_t^\dagger=\sigma_t$ for all $t\in\mathbb{C}$. Indeed, these inner products may be viewed as standard inner products twisted by half-step modular shifts $\sigma_{i/2}$:
    $$\langle X, Y\rangle_L= \langle \sigma_{i/2}\circ X, \sigma_{i/2} \circ Y\rangle_{HS}, \quad \langle X, Y\rangle_R= \langle X\circ \sigma_{-i/2}, Y\circ \sigma_{-i/2}\rangle_{HS}.$$
\end{remark}
\begin{theorem} \label{orthogonal_P_A}
    We have $P_A^{\dagger_L}=P_A$ and $(P'_A)^{\dagger_R}=P_A'$.
\end{theorem}

\begin{proof}
    Using Lemma~\ref{schur_conjugate}, we compute:
    \begin{align*}
        \langle X, P_A Y\rangle_L
        &=\delta^{-2}\langle \sigma_i\circ X, A\cdot_S Y \rangle_{HS} \\
        &=\delta^{-2}\langle(\sigma_i \circ A )\cdot_S (\sigma_i\circ X), Y\rangle_{HS}\\
        &=\delta^{-2}\langle \sigma_i \circ (A\cdot_S X), Y\rangle_{HS}\\
        &=\langle P_A X , Y \rangle_L,
    \end{align*}
    and similarly for the right action,
    \begin{align*}
        \langle X, P_A' Y\rangle_R
        &=\delta^{-2}\langle X\circ \sigma_{-i}, Y\cdot_S A \rangle_{HS}\\
        &=\delta^{-2}\langle(X \circ \sigma_{-i})\cdot_S (A\circ \sigma_{-i}), Y\rangle_{HS}\\
        &=\delta^{-2}\langle(X \cdot_S A)\circ \sigma_{-i}, Y\rangle_{HS}\\
        &=\langle P_A' X, Y\rangle_R.
    \end{align*}
\end{proof}

Let us also observe the following. 

\begin{proposition}
    Let $(\mathcal{M}, \Psi, A)$ be a quantum graph. Then $\pi^\op$ is an isometry into $(B(L^2(\mathcal{M}, \Psi)), \langle\cdot, \cdot \rangle_L)$ and $\delta^{-1}\pi$ is an isometry into $(B(L^2(\mathcal{M}, \Psi)), \langle\cdot, \cdot \rangle_R)$.
\end{proposition}

\begin{proof}
    Let $a,b\in \mathcal{M}^\op$ and let $(\xi_j)$ be an orthonormal basis of $L^2(\mathcal{M}, \Psi)$. Then:
    \begin{align*}
        \langle \pi^\op(a), \pi^\op(b)\rangle_L
        &=\langle \sigma_i \circ \pi^\op(a), \pi^\op(b)\rangle_{HS}\\
        &=\sum_j \langle (\sigma_i\circ \pi^\op(a))\xi_j, \pi^\op(b)\xi_j \rangle_\Psi\\
        &=\sum_j\langle\sigma_i(\xi_j)\sigma_{i/2}(a), \xi_j \sigma_{-i/2}(b)\rangle_\Psi\\
        &=\Psi\left(\sigma_{-i/2}(a^*)\sum_j \sigma_{-i}(\xi_j^*)\xi_j \sigma_{-i/2}(b)\right)\\
        &=\delta^2\Psi(\sigma_{-i/2}(a^*)\sigma_{-i/2}(b))\\
        &=\delta^2\langle a, b \rangle_\Psi.
    \end{align*}
    The second statement follows analogously.
\end{proof}
Let us observe that the edge space $S_L(G)$ is a closed $\pi^{\op}$-bimodule, precisely as required in Weaver's framework. By identifying $\mathcal{M}$ with its image under $\pi_\Psi$ and using the relation
\[
\pi_\Psi(\mathcal{M})'=\pi_\Psi^{\op}(\mathcal{M}^{\op}),
\]
we recover Weaver's original definition of a quantum graph~\cite{NW2015}. By symmetry, an entirely analogous construction can be carried out for the right action, yielding the corresponding dual bimodule structure.

\begin{corollary} \label{M bimodule}
    The right edge space $S_R(G)\subseteq B(L^2(\mathcal{M}, \Psi))$ is a closed $\pi_\Psi(\mathcal{M})$--$\pi_\Psi(\mathcal{M})$ bimodule with respect to the right scalar product.
\end{corollary}

These parallel constructions give rise to two distinct subspaces of $B(L^2(\mathcal{M},\Psi))$, each of which characterizes the same underlying quantum graph via its associated generators. Nevertheless, the spaces $S_L$ and $S_R$ need not coincide in general, as the following example illustrates.

\begin{example}
    Let $\left(\mathrm{Mat}_2, 2 \mathrm{Tr}, A^{(1B)}_{0,0}\right)$ be a unital, tracial quantum graph, where
    $$2A^{(1B)}_{0,0}=\ket{E_{11}}\bra{E_{22}}+\ket{E_{12}}\bra{E_{21}}+\ket{E_{21}}\bra{E_{12}}+\ket{E_{22}}\bra{E_{11}},$$
    or equivalently,
    $$A^{(1B)}_{{0,0}}=\left[\begin{array}{cccc}
        0 & 0 & 0 & 1 \\
        0 & 0 & 1 & 0 \\
        0 & 1 & 0 & 0 \\
        1 & 0 & 0 & 0 \\
    \end{array}\right]$$
    with respect to the standard basis $\{E_{11}, E_{12}, E_{21}, E_{22}\}$. In this picture, we have the identities $\pi(a)=a\otimes \mathrm{id}$ and $\pi^\op(a)=\mathrm{id} \otimes a^T$, which yields:
    $$S_L=\begin{bmatrix}
        0&1\\1&0
    \end{bmatrix}\otimes \mathrm{Mat}_2, \quad S_R=\mathrm{Mat}_2 \otimes \begin{bmatrix}
        0&1\\1&0
    \end{bmatrix}.$$
\end{example}

These results show that the abstract Weaver definition of a graph becomes significantly more concrete if we treat $\mathcal{M}$ as a $*$-subalgebra of $B(L^2(\mathcal{M}, \Psi))$. Indeed, any $\pi(\mathcal{M})' $--$ \pi(\mathcal{M})'$ bimodule $S_L$ is generated by an adjacency operator $A$ satisfying $A \cdot_S A = \delta^2 A$. 

Crucially, this observation shows that, given $A$, we can construct a genuine $\pi(\mathcal{M})$--$\pi(\mathcal{M})$ bimodule, thereby completely avoiding the potentially delicate use of the opposite algebra $\mathcal{M}^{\op}$. This naturally raises the question of whether a quantum graph can be defined directly as a $\pi(\mathcal{M})$--$\pi(\mathcal{M})$ bimodule. To establish this formulation, we need an explicit correspondence that describes how the two representations are transformed into one another via the transposition operations.

\begin{lemma}\label{pi()TL}
We have:
    $$\pi(a)^{TL}=\pi^{\op}(\sigma_{i/2}(a)), \quad \pi(a)^{TR}=\pi^{\op}(\sigma_{-i/2}(a)),$$
    $$(\pi^\op(b))^{TL}=\pi(\sigma_{i/2}(b)), \quad (\pi^\op(b))^{TR}=\pi(\sigma_{-i/2}(b)),$$
    for all $a\in \mathcal{M}$ and $b\in \mathcal{M}^\op$.
\end{lemma}

\begin{proof}
    Let $(\xi_j)_{j}$ be an orthonormal basis of $L^2(\mathcal{M}, \Psi)$. Then:
    \begin{align*}
        (\pi(a))^{TL}&=\sum_j \left(\ket{\xi_j}\bra{a^*\xi_j}\right)^{TL}=\sum_j \ket{\xi_j^*a}\bra{\sigma_{-i}(\xi_j^*)}=\pi^\op(\sigma_{i/2}(a))\sum_j \ket{\xi_j^*}\bra{\sigma_{-i}(\xi_j^*)}=\pi^\op(\sigma_{i/2}(a)),
    \end{align*}
    \begin{align*}
        (\pi(a))^{TR}&=(\pi(a^*)^{TL})^{\dagger}=\pi^\op(\sigma_{i/2}(a^*))^\dagger=\pi^\op(\sigma_{-i/2}(a)),
    \end{align*}
    \begin{align*}
        (\pi^\op(b))^{TR}&=(\pi(\sigma_{-i/2}(b))^{TL})^{TR}=\pi(\sigma_{-i/2}(b)),
    \end{align*}
    and
    \begin{align*}
        (\pi^\op(b))^{TL}&=(\pi(\sigma_{i/2}(b))^{TR})^{TL}=\pi(\sigma_{i/2}(b)).
    \end{align*}
\end{proof}

This structural correspondence provides a direct identification between the left and right edge spaces through the generalized transposition map, showing that these spaces are isometrically isomorphic.

\begin{proposition}\label{S_R-S_L-equivalence}
    Let $(\mathcal{M}, \Psi, A)$ be a quantum graph with some $\delta$, and let $S_L (S_R)$ be the left (right) edge space, respectively. Then $\overline{S_L}:=(S_L^\dagger)^{TL}=S_R$ and $\overline{S_R}=S_L$.
\end{proposition}

\begin{proof}
    The proposition is a direct consequence of Lemma~\ref{pi()TL} and the fact that $A$ is real.
\end{proof}

As an immediate consequence of this equivalence, we get the following corollary.

\begin{corollary}
    Any $\pi(\mathcal{M})$--$\pi(\mathcal{M})$ bimodule $S_R\subseteq B(L^2(\mathcal{M}, \Psi))$ defines a $\pi^\op(\mathcal{M}^\op)$--$\pi^\op(\mathcal{M}^\op)$ bimodule $S_L$ via $S_L=\overline{S_R}$.
\end{corollary}

The following theorem establishes that every abstract $\pi(\mathcal{M})$--$\pi(\mathcal{M})$ bimodule arises uniquely from a $*$-preserving, Schur-idempotent adjacency operator, thereby showing equivalence of two frameworks.

\begin{theorem}
    Let $(\mathcal{M}, \Psi)$ be a quantum system with some $\delta$, and let $S_R\subseteq B(L^2(\mathcal{M}, \Psi))$ be any $\pi(\mathcal{M}) $--$ \pi(\mathcal{M})$ bimodule. Then there exists a unique $*$-preserving operator $A\in B(L^2(\mathcal{M}, \Psi))$ such that $A\cdot_S A=\delta^2 A$ and $S_R=\pi(\mathcal{M})A\pi(\mathcal{M})$.
\end{theorem}

\begin{proof}
    By the previous corollary, any $\pi(\mathcal{M})$--$\pi(\mathcal{M})$ bimodule $S_R$ defines a $\pi^\op(\mathcal{M}^\op)$--$\pi^\op(\mathcal{M}^\op)$ bimodule $\overline{S_R}$. Hence, $\overline{S_R}$ defines, by \cite[Section 2.2]{daws2025quantumgraphsinfinitedimensionshilbertschmidts}, a unique operator $A$ such that $A\cdot_S A=\delta^2 A$, $A$ is $*$-preserving, and $\overline{S_R}=\pi^\op(\mathcal{M}^\op)A \pi^\op(\mathcal{M}^\op)$. By Proposition~\ref{S_R-S_L-equivalence}, $A$ satisfies the desired properties.
\end{proof}

\begin{remark}
Note that, in general, the operator $A$ obtained from the preceding theorem need not be an adjacency operator in the sense of Definition~\ref{q-graph:definition}, since it may fail to be real. However, this condition is automatically satisfied whenever $S_R$ (or equivalently $S_L$) is invariant under the modular automorphism group $\sigma_t$, as this invariance implies that $A$ commutes with $\sigma_t$. We further observe that the reality condition is trivially satisfied for tracial graphs, for which $\sigma_t$ is the identity map, and for undirected graphs, where $S_R=S_R^\dagger$ implies $A=A^\dagger$, so that $A^\dagger$ is automatically $*$-preserving.
\end{remark}

\section{Quantum incidence map}
\label{sec:incidence}
In Section~2 (see Example~\ref{classical incidence matrix}), we recalled the classical incidence matrix and showed how it maps the edge space of a graph to its vertex space, ultimately giving rise to the combinatorial Laplacian for both directed and undirected graphs.

Here, we extend these ideas to the quantum setting. Using the left and right edge spaces $S_L$ and $S_R$ discussed
in the previous section, we introduce the corresponding quantum incidence operators. To account for the non-commutative structure of the underlying algebra and to ensure compatibility between the left and right actions of the incidence operators, our definitions incorporate the modular automorphism group $\sigma_t$. We show that, under a suitable identification of the edge spaces, this construction naturally yields the quantum combinatorial Laplacian introduced in \cite[Lemma~2.6]{Matsuda_2024}.

\begin{definition}
\label{def:incidence_quantum}
Let $G=(\mathcal{M}, \Psi, A)$ be a quantum graph and let $S_L (S_R)$ be its left (right) edge space. We define two, left and right, incidence operators $\mathcal{K}_L:\mathcal{M}^\op\to S_L(G)$ and $\mathcal{K}_R:\mathcal{M}\to S_R(G)$ by
$$\mathcal{K}_L(a):=\delta^{-2}(\pi^\op(a)A-A\pi^\op(a))$$
$$\mathcal{K}_R(a):=\delta^{-2}(\pi(a)A-A\pi(a))$$
\end{definition}
\begin{remark}
Clearly, for classical graphs representations $\pi$ and $\pi^\op$ are identical, and the operators $\mathcal{K}_L$ and $\mathcal{K}_R$ coincide, fully agreeing with the incidence matrix defined in Example~\ref{classical incidence matrix}. The fact that the ranges of $\mathcal{K}_L$ ($\mathcal{K}_R$) is $S_L$ ($S_R$) comes from the following lemma.
\end{remark}

\begin{lemma}
For any quantum graph $(\mathcal{M}, \Psi, A)$, we have the following identity.
    $$\mathcal{K}_L(a)=\delta^{-2}A\cdot_S\left(\ket{\sigma_{-i/2}(a)}\bra{\I}-\ket{\I}\bra{\sigma_{-i/2}(a^*)}\right),$$
$$\mathcal{K}_R(a)=\delta^{-2}\left(\ket{a}\bra{\I}-\ket{\I}\bra{a^*}\right)\cdot_S A.$$
\end{lemma}
\begin{proof}
    The lemma is a consequence of Remark~\ref{cdotpi}.
\end{proof}

Analogous to classical incidence matrices, which act as boundary operators encoding the geometry of the underlying graph, their quantum counterparts must respect the algebraic structure of the underlying von Neumann algebra. As an immediate consequence of Definition~\ref{def:incidence_quantum}, we have the following theorem.
\begin{theorem}\label{leibnitz_rule}
For any quantum graph $(\mathcal{M}, \Psi, A)$, the incidence operators satisfy the Leibniz rule:
$$\mathcal{K}_L(a \cdot_\op b)=\pi_\Psi^\op(a)\mathcal{K}_L(b)+\mathcal{K}_L(a)\pi^\op_\Psi(b),$$
$$\mathcal{K}_R(ab)=\pi_\Psi(a)\mathcal{K}_R(b)+\mathcal{K}_R(a)\pi_\Psi(b).$$
\end{theorem}

To construct the quantum combinatorial Laplacian in analogy with the classical setting, we first need to determine the adjoints of $\mathcal{K}_L$ and $\mathcal{K}_R$ with respect to the adapted Hilbert space inner products $\langle\cdot,\cdot\rangle_L$ and $\langle\cdot,\cdot\rangle_R$ introduced in the previous section.

\begin{lemma}
For any quantum graph $G=(\mathcal{M}, \Psi, A)$, the Hermitian conjugates ${}^{\dagger_L}$ and ${}^{\dagger_R}$ of the left and right incidence operators are given by
$$(\mathcal{K}_L)^{\dagger_L}(X)=(\sigma_{i/2}\circ X-\sigma_{-i/2}\circ X^{TL})\I, \qquad (\mathcal{K}_R)^{\dagger_R}(X)=(X-\sigma_i\circ X^{TR})\I.$$
\end{lemma}

\begin{proof}
Let $X\in S_L$ and $a\in \mathcal{M}^\op$. Using Remark~\ref{cdotpi}, Proposition~\ref{Transpose_features}, and Theorem~\ref{orthogonal_P_A}, we compute:
\begin{align*}
\langle X, \mathcal{K}_L(a)\rangle_{L}
&=\delta^{-2}\left\langle X,A\cdot_S\left(\ket{\sigma_{-i/2}(a)}\bra{\I}-\ket{\I}\bra{\sigma_{-i/2}(a^*)}\right) \right\rangle_{L}\\ 
&=\delta^{-2}\left\langle A\cdot_S X,\ket{\sigma_{-i/2}(a)}\bra{\I}-\ket{\I}\bra{\sigma_{-i/2}(a^*)} \right\rangle_{L}\tag{1} \\
&=\left\langle X,\ket{\sigma_{-i/2}(a)}\bra{\I}-\ket{\I}\bra{\sigma_{-i/2}(a^*)} \right\rangle_{L}\tag{2} \\
&=\mathrm{Tr}((\sigma_i\circ X)^\dagger \ket{\sigma_{-i/2}(a)}\bra{\I})
-\mathrm{Tr}((\sigma_i\circ X)^\dagger \ket{\I}\bra{\sigma_{-i/2}(a^*)})\\
&=\mathrm{Tr}((\sigma_i\circ X)^\dagger \ket{\sigma_{-i/2}(a)}\bra{\I})
-\mathrm{Tr}\left([(\sigma_i\circ X)^\dagger \ket{\I}\bra{\sigma_{-i/2}(a^*)}]^{TR}\right)\tag{3} \\
&=\mathrm{Tr}((\sigma_i\circ X)^\dagger \ket{\sigma_{-i/2}(a)}\bra{\I})
-\mathrm{Tr}\left(\ket{\sigma_{-i/2}(a)}\bra{\I}((\sigma_i\circ X)^\dagger)^{TR}\right)\tag{4} \\
&=\mathrm{Tr}((\sigma_{i}\circ X)^\dagger \ket{\sigma_{-i/2}(a)}\bra{\I})
-\mathrm{Tr}\left(\ket{\sigma_{-i/2}(a)}\bra{\I}(X^{TL}\circ \sigma_{-i})^\dagger\right)\tag{5}\\
&=\mathrm{Tr}((\sigma_{i}\circ X)^\dagger \ket{\sigma_{-i/2}(a)}\bra{\I})
-\mathrm{Tr}\left((X^{TL}\circ \sigma_{-i})^\dagger \ket{\sigma_{-i/2}(a)}\bra{\I}\right)\\
&=\langle \I, (\sigma_i\circ X-X^{TL}\circ \sigma_{-i})^\dagger \sigma_{-i/2}(a)\rangle_\Psi\tag{6}\\
&=\langle(\sigma_i \circ X-X^{TL}\circ \sigma_{-i})\I,\sigma_{-i/2}(a)\rangle_\Psi\\
&=\langle\sigma_{-i/2}[(\sigma_i \circ X-X^{TL}\circ \sigma_{-i})\I], a\rangle_\Psi\\
&=\langle(\sigma_{i/2}\circ X-\sigma_{-i/2}\circ X^{TL})\I, a\rangle_{\Psi}.\tag{7}
\end{align*}

In steps (1) and (2), we used that $P_A(\cdot)=\delta^{-2}A\cdot_S(\cdot)$ is self-adjoint with respect to the inner product $\langle\cdot,\cdot\rangle_L$, together with the fact that $X\in\operatorname{ran} P_A$.

In steps (3), (4), and (5), we used the identities
\[
\operatorname{Tr}(XY)=\operatorname{Tr}((XY)^{TR})
=\operatorname{Tr}(Y^{TR}X^{TR})
=\operatorname{Tr}(X^{TR}Y^{TR}),
\]
valid for all $X,Y\in B(L^2(\mathcal{M},\Psi))$, as well as the relation
\[
(X^\dagger)^{TR}=(X^{TL})^\dagger.
\]
Furthermore, we used the identity
\[
(\sigma_i\circ X)^{TL}=X^{TL}\circ\sigma_{-i},
\]
which follows from the fact that $\sigma_i^{TL}=\sigma_{-i}$.

In step (6), we used the cyclicity of the trace. Finally, in step (7), we used the self-adjointness of $\sigma_i$ together with the identity $\sigma_t(\mathbb{I})=\mathbb{I}$.

Using analogous properties, for any $Y\in S_R$ and $b\in \mathcal{M}$, we compute:

\begin{align*}
\langle Y,\mathcal{K}_R(b)\rangle_{R}
&=\left \langle Y,\ket{b}\bra{\I}-\ket{\I}\bra{b^*}\right\rangle_{R}\\
&=\mathrm{Tr}((Y\circ \sigma_{-i})^\dagger \ket{b}\bra{\I})
-\mathrm{Tr}((Y\circ \sigma_{-i})^\dagger \ket{\I}\bra{b})\\
&=\mathrm{Tr}((Y\circ \sigma_{-i})^\dagger \ket{b}\bra{\I})
-\mathrm{Tr}\left(\ket{b}\bra{\I}(( Y\circ \sigma_{-i})^\dagger)^{TL}\right)\\
&=\mathrm{Tr}((Y\circ \sigma_{-i})^\dagger \ket{b}\bra{\I})
-\mathrm{Tr}((\sigma_i\circ Y^{TR})^\dagger \ket{b}\bra{\I})\\
&=\langle \I,(Y-\sigma_i\circ Y^{TR})^\dagger b\rangle_\Psi\\
&=\langle (Y-\sigma_i\circ Y^{TR})\I,b\rangle_\Psi.
\end{align*}

\end{proof}
Having established explicit formulas for the adjoints of the incidence operators, we can now introduce the central operator-theoretic objects of this section. Composing the incidence operators with their corresponding left and right adjoints yields the natural quantum counterparts of the classical combinatorial Laplacian.

\begin{definition}
The left Laplacian $\Delta_L:\mathcal{M}^\op\to \mathcal{M}^\op$ and the right Laplacian $\Delta_R:\mathcal{M}\to \mathcal{M}$ of a graph $G=(\mathcal{M},\Psi, A)$ are defined by
$$\Delta_L:=\frac{\mathcal{K}_L^{\dagger_L}\mathcal{K}_L}{2}, \qquad \Delta_R:=\frac{(\mathcal{K}_R)^{\dagger_R}\mathcal{K}_R}{2}.$$
\end{definition}
Although the definitions of $\Delta_L$ and $\Delta_R$ are naturally formulated in an abstract terms, explicit formulas are essential for applications and spectral analysis. The following theorem expresses these Laplacians directly in terms of the adjacency operator $A$ and the associated degree-like elements obtained from the action of $A$ on the identity.
\begin{theorem}
For any quantum graph $G=(\mathcal{M}, \Psi, A)$, we have
$$2\Delta_L(v)=\delta^{-2}\left((A^\dagger\I)\cdot_\op v+v\cdot_\op (A\I)-Av-A^\dagger v\right),$$
and
$$2\Delta_R(v)=\delta^{-2}\left((A^\dagger\I)\cdot v+v\cdot (A\I)-Av-A^\dagger v\right).$$
\end{theorem}
\begin{proof}
Using Lemma~\ref{sigma_commuting}, we compute:
\begin{align*}
\delta^2(\mathcal{K}_L)^{\dagger_L} \mathcal{K}_L(v)
&=(\mathcal{K}_L)^{\dagger_L}(\pi^\op(v)A-A\pi^\op(v))\\
&=\sigma_{i/2}\circ(\pi^\op(v)A-A\pi^\op(v))\I-\sigma_{-i/2}(\pi^\op(v)A-A\pi^\op(v))^{TL}\I\\
&=\sigma_{i/2}\circ(\pi^\op(v)A-A\pi^\op(v))\I-\sigma_{-i/2}(A^\dagger \pi(\sigma_{i/2}(v))-\pi(\sigma_{i/2}(v))A^\dagger)\I\\
&=\sigma_{i/2}((A\I)\sigma_{-i/2}(v))-\sigma_{i/2}(A\sigma_{-i/2}(v))-\sigma_{-i/2}(A^\dagger \sigma_{i/2}(v))+\sigma_{-i/2}(\sigma_{i/2}(v) (A^\dagger \I))\\
&=(A\I)v+v(A^{\dagger}\I)-Av-A^\dagger v.
\end{align*}
Similarly,
\begin{align*}
\delta^2(\mathcal{K}_R)^{\dagger_R} \mathcal{K}_R(v)
&=(\mathcal{K}_R)^{\dagger_R}(\pi(v)A-A\pi(v))\\
&=(\pi(v)A-A\pi(v))\I-\sigma_i(A^\dagger \pi^\op(\sigma_{-i/2}(v))-\pi^\op(\sigma_{-i/2}(v))A^\dagger)\I\\
&=v(A\I)-Av-\sigma_i (A^\dagger \sigma_{-i}(v))+\sigma_i((A^\dagger \I)\sigma_{-i}(v))\\
&=v(A\I)+(A^{\dagger}\I)v-Av-A^\dagger v.
\end{align*}
\end{proof}
We conclude this section by establishing that the left and right Laplacians are antiunitary equivalent.
\begin{theorem}
    Let $(\mathcal{M},\Psi,A)$ be a quantum graph. Then the left and right Laplacians satisfy
    \[
    \Delta_L=J\Delta_RJ,
    \]
    where $J:\mathcal{M}\to\mathcal{M}$ is the antiunitary operator defined by
    \[
    J(v)=\sigma_{-i/2}(v^*).
    \]
\end{theorem}

\begin{proof}
Since both $A$ and $A^\dagger$ commute with the involution and the modular automorphism group $\sigma_t$, the antiunitary operator $J$ commutes with both operators. Hence,
\[
[A,J]=0.
\]
Consequently, $J$ intertwines the left and right degree actions, giving
\[
J\pi(A\mathbb{I})=\pi^{\op}(A\I)J,
\]
and similarly for $A^\dagger$. These intertwining relations directly imply
\[
J\Delta_LJ=\Delta_R.
\]
\end{proof}

Since both $\Delta_L$ and $\Delta_R$ are self-adjoint operators with real spectra, we immediately obtain the following consequence.

\begin{corollary}
For any quantum graph $G=(\mathcal{M},\Psi,A)$, the left and right Laplacians have identical spectra:
\[
\operatorname{spec}(\Delta_L)=\operatorname{spec}(\Delta_R).
\]
\end{corollary}

\section{Comparison to Matsuda's Laplacian and incidence operator}
\label{sec:matsuda}

In \cite{Matsuda_2024}, a different notion of incidence operator (Definition~2.1), together with the corresponding Laplacian, was introduced. In this section, we provide a brief overview of that framework and conclude with a comparison between this construction and ours.

\begin{definition}
Let $(\mathcal{M}, \Psi, A)$ be a quantum graph. The Matsuda incidence operator 
$$\nabla_A : \mathcal{M} \to \mathcal{M} \otimes \mathcal{M}$$
is defined by
$$\nabla_A(x):=\delta^{-2}(A^\dagger \otimes \id-\id \otimes A)m^\dagger (x).$$
\end{definition}

Matsuda's formulation realizes the edge space inside the algebraic tensor product $\mathcal{M}\otimes\mathcal{M}$, whereas our framework describes it in terms of the operator space $B(L^2(\mathcal{M},\Psi))$. To establish a precise correspondence between these two viewpoints and construct a rigorous dictionary between the respective theories, we introduce a canonical linear identification map.

\begin{definition}\label{tenop}
The linear map $\zeta:B(L^2(\mathcal{M},\Psi))\to \mathcal{M}\otimes \mathcal{M}$ is given by
$$\zeta(T):=(\id \otimes T)m^\dagger(\I), \quad \text{for all } T\in B(L^2(\mathcal{M},\Psi)).$$
\end{definition}

The utility of this map lies in its clean action on rank-one operators, which allows for a straightforward inversion, as demonstrated below.

\begin{lemma} \label{zeta}
We have
$$\zeta(\ket{x}\bra{y})=\sigma_{-i}(y^*) \otimes x \quad \text{and} \quad \zeta^{-1}(y \otimes x)=\ket{x}\bra{\sigma_{-i}(y^*)}.$$
\end{lemma}

\begin{proof}
Let $(\xi_j)$ be any orthonormal basis of $L^2(\mathcal{M}, \Psi)$. Recall that
$$m^\dagger(\I)=\sum_j \xi_j \otimes \xi_j^*=\sum_j\sigma_{-i}(\xi_j^*)\otimes \xi_j.$$
By direct computation,
\begin{align*}
\zeta(\ket{x}\bra{y})
&=(\id \otimes \ket{x}\bra{y})m^\dagger(\I)\\
&=\sum_j \left(\xi_j\otimes \langle y, \xi_j^*\rangle x\right)\\
&=\sum_j \langle \xi_j, \sigma_{-i}(y^*)\rangle \, \xi_j \otimes x\\
&=\sigma_{-i}(y^*) \otimes x.
\end{align*}
The formula for $\zeta^{-1}$ follows immediately.
\end{proof}

Because the construction of a quantum combinatorial Laplacian fundamentally depends on the adjoints of $\mathcal{K}$ and $\nabla_A$, we must first determine how the tensor-to-operator identification from Definition~\ref{tenop} is compatible with the underlying Hilbert space structures. The following lemma computes the relevant Hermitian adjoints explicitly.

\begin{lemma}\label{zeta-1}
We have
$$\zeta^\dagger(y\otimes x)=\ket{\sigma_{-i}(x)}\bra{y^*}, \quad (\zeta^{-1})^\dagger(\ket{x}\bra{y})=y^* \otimes \sigma_i(x).$$
\end{lemma}

\begin{proof}
Let $(\xi_j)$ be any orthonormal basis of $L^2(\mathcal{M},\Psi)$ and let $x,y,a,b\in \mathcal{M}$. By direct computation,
\begin{align*}
\langle \zeta^\dagger(y \otimes x), \ket{a}\bra{b}\rangle_L
&=\langle y \otimes x, \zeta (\ket{a}\bra{b})\rangle_{\Psi \otimes \Psi}\\
&=\langle y\otimes x, \sigma_{-i}(b^*)\otimes a\rangle_{\Psi \otimes \Psi}\\
&=\langle y, \sigma_{-i}(b^*)\rangle_\Psi \langle x, a \rangle_\Psi\\
&=\langle b, y^* \rangle_\Psi \langle x, a \rangle_\Psi\\
&=\mathrm{Tr}(\ket{y^*}\bra{x} \ket{a}\bra{b})\\
&=\mathrm{Tr}\big((\ket{y^*}\bra{\sigma_{-i}(x)}\circ \sigma_i) \ket{a}\bra{b}\big)\\
&=\langle \ket{\sigma_{-i}(x)}\bra{y^*}, \ket{a}\bra{b}\rangle_L.
\end{align*}
Moreover,
\begin{align*}
\langle(\zeta^{-1})^\dagger(\ket{x}\bra{y}), b \otimes a \rangle_{\Psi\otimes \Psi }
&= \langle \ket{x}\bra{y}, \zeta^{-1}(b \otimes a)\rangle_L\\
&=\langle \ket{x}\bra{y}, \ket{a}\bra{\sigma_{-i}(b^*)}\rangle_L\\
&=\langle\ket{\sigma_i(x)}\bra{y}, \ket{a}\bra{\sigma_{-i}(b^*)}\rangle_{HS}\\
&=\mathrm{Tr}(\ket{y}\bra{\sigma_{i}(x)} \ket{a}\bra{\sigma_{-i}(b^*)})\\[1ex]
&=\langle\sigma_i(x), a\rangle_\Psi \langle \sigma_{-i}(b^*), y \rangle_\Psi\\
&=\langle y^*, b \rangle_\Psi \langle\sigma_i(x), a \rangle_\Psi\\
&=\langle y^*\otimes \sigma_i(x), b \otimes a \rangle_{\Psi \otimes \Psi.}
\end{align*}
\end{proof}

Several properties of Matsuda's incidence operator are developed in \cite[Chapter~2]{Matsuda_2024}. Despite originating from distinct geometric viewpoints—one formulated through tensorial operations and the other through commutator expressions—the two incidence operators are intimately connected. The following proposition establishes that our operator $\mathcal{K}_L$ corresponds to Matsuda's operator $\nabla_A$ under a specific identification map.

\begin{proposition}The following holds:
$$\mathcal{K}_L=\zeta^{-1}\circ \nabla_A\circ \sigma_{-i/2}.$$
\end{proposition}

\begin{proof}
From the proof of Theorem~\ref{leibnitz_rule}, we know that
$$\delta^2\mathcal{K}_L(x)=\pi^\op(x)A-A\pi^\op(x).$$
On the other hand,
\begin{align*}
\delta^2\zeta^{-1}(\nabla_A (x))
&=\zeta^{-1}(A^\dagger\otimes \id-\id \otimes A )m^\dagger(x) \\
&=\zeta^{-1}\Bigl(\sum\limits_j\bigl(A^\dagger \xi_j\otimes \xi_j^*x-\xi_j\otimes A \xi_j^* x\bigr)\Bigr)\\
&=\sum\limits_j\ket{\xi_j^* x}\bra{\sigma_{-i}((A^\dagger \xi_j)^*)}
-\ket{A \xi_j^*x}\bra{\sigma_{-i}(\xi_j^*)}\\
&=\sum\limits_j\rho(x)\ket{\xi_j^*}\bra{\sigma_{-i}(\xi_j^*)}A
-A\rho(x)\ket{\xi_j^*}\bra{\sigma_{-i}(\xi_j^*)}\\
&=\rho(x)A-A\rho(x),
\end{align*}
where $\rho(x)y=yx$ for all $x,y\in \mathcal{M}$.
Since $\rho \circ \sigma_{-i/2}=\pi^\op$, the proof is complete.
\end{proof}

This factorization allows us to establish a precise correspondence between the two frameworks and demonstrates that, despite their different formulations, they yield the same Laplacian operator.

First, note that by Lemmas~\ref{zeta} and \ref{zeta-1}, we have
$$(\zeta^{-1})^\dagger \zeta^{-1}=\sigma_i\otimes \sigma_i.$$
Moreover, since $A$ and $A^\dagger$ commute with all $\sigma_t$, we obtain
$$\mathcal{K}_L=\zeta^{-1}\circ(\sigma_{-i/2} \otimes \sigma_{-i/2})\circ \nabla_A.$$
We conclude by computing :
\begin{align*}
\mathcal{K}_L^\dagger \mathcal{K}_L
&=\nabla_A^\dagger (\sigma_{-i/2}\otimes \sigma_{-i/2})(\zeta^{-1})^\dagger \zeta^{-1}(\sigma_{-i/2}\otimes\sigma_{-i/2}) \nabla_A\\
&=\nabla_A^\dagger(\sigma_{-i/2}\otimes \sigma_{-i/2})(\sigma_{i}\otimes \sigma_{i})(\sigma_{-i/2}\otimes \sigma_{-i/2})\nabla_A\\
&=\nabla_A^\dagger \nabla_A.
\end{align*}

\section{Examples of the quantum Laplacian of tracial \texorpdfstring{$M_2$}{M2} quantum graphs}
\label{sec:examples}

In this section, we compute the Laplacians associated with quantum graphs defined on the $\mathrm{C}^*$-algebra $\mathrm{Mat}_2$ of $2\times2$ complex matrices. A classification of all such graphs, including both tracial and non-tracial cases, is given in \cite{Kiefer_2026}. Here, we restrict our attention to the tracial setting. We specify here to the left Laplacians $\Delta_L$ and denote them by $\Delta$.

Direct spectral analysis of operators acting on the $\mathrm{C}^*$-algebra $\mathrm{Mat}_2$ can be technically involved. To simplify the computations and render the spectral analysis more explicit, we identify $2\times2$ matrices with column vectors in $\mathbb{C}^4$.

\begin{definition}
    We define the linear map $\kappa:\Mat_2 \ni E_{ij}\mapsto e_{(i-1)2+j}\in \CC^4$, where $\{e_k\}$ is the canonical basis of $\CC^4$. In particular,
    \[ \kappa\left(\left[\begin{array}{cc}
        X_{11} & X_{12} \\
        X_{21} & X_{22} \\
    \end{array}\right]\right)=\left(\begin{array}{c}
         X_{11}  \\
         X_{12}  \\
         X_{21}  \\
         X_{22}  \\
    \end{array}\right). \]
    Note that $\kappa$ is an isometry with respect to the Hilbert-Schmidt inner product on $\Mat_2$ and the canonical inner product on $\CC^4$.
\end{definition}

To use the isometry $\kappa$ effectively in our spectral analysis, it is necessary to determine how it intertwines the algebraic operations of $\mathrm{Mat}_2$, in particular matrix multiplication and the adjoint operation, with their corresponding matrix representations on $\mathbb{C}^4$. The following two lemmas provide this dictionary explicitly.

\begin{lemma}
    For any $X, Y\in \Mat_2$, we have the following identities:
    \[ \kappa(XY)=(X\otimes \I)\kappa(Y)=(\I \otimes Y^T)\kappa(X). \]
\end{lemma}

\begin{proof}
    Let $X=\sum_{i,j}X_{ij}E_{ij}$ and $Y=\sum_{i,j}Y_{ij}E_{ij}$. By direct computation, we obtain
     \begin{align*}
        \kappa(XY)&=\kappa\left(\sum_{i,j,k} X_{ik} Y_{kj} E_{ij}\right)=\left(\begin{array}{c}
             X_{11}Y_{11}+X_{12}Y_{21}  \\
             X_{11}Y_{12}+X_{12}Y_{22}  \\
             X_{21}Y_{11}+X_{22}Y_{21}  \\
             X_{21}Y_{12}+X_{22}Y_{22}  \\
        \end{array}\right).
    \end{align*} 
    Furthermore, 
     \begin{align*}
        (X\otimes\I)\kappa(Y)&=\left[\begin{array}{cccc}
            X_{11} & 0 & X_{12} & 0 \\
            0 & X_{11} & 0 & X_{12} \\
            X_{21} & 0 & X_{22} & 0 \\
            0 & X_{21} & 0 & X_{22} \\
        \end{array}\right] \left(\begin{array}{c}
             Y_{11}  \\
             Y_{12}  \\
             Y_{21}  \\
             Y_{22}  \\
        \end{array}\right)=\left(\begin{array}{c}
             X_{11}Y_{11}+X_{12}Y_{21}  \\
             X_{11}Y_{12}+X_{12}Y_{22}  \\
             X_{21}Y_{11}+X_{22}Y_{21}  \\
             X_{21}Y_{12}+X_{22}Y_{22}  \\
        \end{array}\right),
    \end{align*} 
and finally,
\begin{align*}
    (\I \otimes Y^T)\kappa(X)&=\left[\begin{array}{cccc}
        Y_{11} & Y_{21} & 0 & 0 \\
        Y_{12} & Y_{22} & 0 & 0 \\
        0 & 0 & Y_{11} & Y_{21} \\
        0 & 0 & Y_{12} & Y_{22} \\
    \end{array}\right] \left(\begin{array}{c}
         X_{11}  \\
         X_{12}  \\
         X_{21}  \\
         X_{22}  \\
    \end{array}\right)=\left(\begin{array}{c}
             X_{11}Y_{11}+X_{12}Y_{21}  \\
             X_{11}Y_{12}+X_{12}Y_{22}  \\
             X_{21}Y_{11}+X_{22}Y_{21}  \\
             X_{21}Y_{12}+X_{22}Y_{22}  \\
        \end{array}\right).
\end{align*} 
\end{proof}

\begin{lemma}
    For any $B=\sum_{i,j} B_{ij}E_{ij}\in \Mat_4$, we have 
    \[ \left(\kappa^{-1}(B^\dagger\kappa(\I))\right)^T=\left[\begin{array}{cc}
            \overline{B_{11}}+\overline{B_{41}} &  \overline{B_{13}}+\overline{B_{43}}\\
            \overline{B_{12}}+\overline{B_{42}} & \overline{B_{14}}+\overline{B_{44}}
        \end{array}\right] \]
        and
        \[ \kappa^{-1}(B\kappa(\I))=\left[\begin{array}{cc}
            B_{11}+B_{14} & B_{21}+B_{24} \\
            B_{31}+B_{34} & B_{41}+B_{44}
        \end{array}\right]. \]
\end{lemma}

\begin{proof}
    By direct computation,
     \begin{align*}
        \left(\kappa^{-1}(B^{\dagger}\kappa(\I))\right)^T&= \left(\kappa^{-1}\left(\begin{array}{c}
             \overline{B_{11}}+\overline{B_{41}}  \\
             \overline{B_{12}}+\overline{B_{42}}  \\
             \overline{B_{13}}+\overline{B_{43}}  \\
             \overline{B_{14}}+\overline{B_{44}}  \\
        \end{array}\right)\right)^T=\left[\begin{array}{cc}
            \overline{B_{11}}+\overline{B_{41}} &  \overline{B_{13}}+\overline{B_{43}}\\
            \overline{B_{12}}+\overline{B_{42}} & \overline{B_{14}}+\overline{B_{44}}
        \end{array}\right].
    \end{align*} 
    The proof of the second identity is analogous.
\end{proof}

The following corollary is a straightforward consequence of the above lemmas.

\begin{corollary} 
    Let $(\Mat_2, 2\mathrm{Tr}, A)$ be a quantum graph. Then its Laplacian $\Delta$ is given by the formula
    \[ \Delta=\kappa^{-1}(\Delta_{\kappa A \kappa^{-1}})\kappa. \]
    In particular, $\mathrm{spec}(\Delta)=\mathrm{spec}(\Delta_{\kappa A \kappa^{-1}})$. Because 
    \[2\Delta(v)=(A\I)v+v(A^\dagger \I)-A-A^\dagger\]
    we have
    \begin{equation}\label{k-k:laplacian}
    2\Delta_{\kappa A \kappa^{-1}}=(A\I)\otimes \id + \id \otimes (A^\dagger \I)^T-\kappa A \kappa^{-1}-\kappa A^\dagger \kappa^{-1}
    \end{equation}
\end{corollary}

\begin{remark}
    Note that because $2\Tr$ has a trivial automorphism group, all tracial graphs presented in \cite{Kiefer_2026} are also graphs in the sense of Definition~\ref{q-graph:definition}. Hence, their Laplacians are given by formula \eqref{k-k:laplacian}.
\end{remark}

\begin{lemma}\label{Laplacian:formula}
    For any $B=\sum_{i,j}B_{ij}E_{ij}$, we have
    \[
    2\Delta_B=
    \resizebox{0.9\textwidth}{!}{ $
    \left[\begin{array}{cccc}
        B_{14} + \overline{B_{41}} & \overline{B_{13}} + \overline{B_{43}} - B_{12} - \overline{B_{21}} & B_{21} + B_{24} - B_{13} - \overline{B_{31}} & -B_{14} - \overline{B_{41}} \\
        \overline{B_{42}} - B_{21} & B_{11} + B_{14} + \overline{B_{14}} + \overline{B_{44}} - B_{22} - \overline{B_{22}} & -B_{23} - \overline{B_{32}} & B_{21} - \overline{B_{42}} \\
        B_{34} - \overline{B_{13}} & -B_{32} - \overline{B_{23}} & B_{41} + B_{44} + \overline{B_{11}} + \overline{B_{41}} - B_{33} - \overline{B_{33}} & \overline{B_{13}} - B_{34} \\
        -B_{41} - \overline{B_{14}} & B_{31} + B_{34} - B_{42} - \overline{B_{24}} & \overline{B_{12}} + \overline{B_{42}} - B_{43} - \overline{B_{34}} & B_{41} + \overline{B_{14}}
    \end{array}\right]
    $ }
    \]
\end{lemma}

\begin{proof}
    Again, by direct computation, we obtain
    \[
    \resizebox{\textwidth}{!}{ $
    \begin{aligned}
        2\Delta_B &= \left[\begin{array}{cccc}
            B_{11}+B_{14} & 0 & B_{21}+B_{24} & 0 \\
            0 & B_{11}+B_{14} & 0 & B_{21}+B_{24} \\
            B_{31}+B_{34} & 0 & B_{41}+B_{44} & 0 \\
            0 & B_{31}+B_{34} & 0 & B_{41}+B_{44} \\
        \end{array}\right]\\
        &\quad +\left[\begin{array}{cccc}
            \overline{B_{11}}+\overline{B_{41}} & \overline{B_{13}}+\overline{B_{43}} & 0 & 0 \\
             \overline{B_{12}}+\overline{B_{42}} & \overline{B_{14}}+\overline{B_{44}} & 0 & 0 \\
             0 & 0 & \overline{B_{11}}+\overline{B_{41}} & \overline{B_{13}}+\overline{B_{43}}\\
             0 & 0 & \overline{B_{12}}+\overline{B_{42}} & \overline{B_{14}}+\overline{B_{44}}
        \end{array}\right]\\
        &\quad -\left[\begin{array}{cccc}
            B_{11}+\overline{B_{11}} & B_{12}+\overline{B_{21}} &B_{13}+\overline{B_{31}} &B_{14}+\overline{B_{41}} \\
            B_{21}+\overline{B_{12}} & B_{22}+\overline{B_{22}} & B_{23}+\overline{B_{32}} &B_{24}+\overline{B_{42}}\\
            B_{31}+\overline{B_{13}} & B_{32}+\overline{B_{23}} & B_{33}+\overline{B_{33}} & B_{34}+\overline{B_{43}}\\
            B_{41}+\overline{B_{14}} & B_{42}+\overline{B_{24}} & B_{43}+\overline{B_{34}} & B_{44}+\overline{B_{44}}\\
        \end{array}\right]\\
        &= \left[\begin{array}{cccc}
            B_{14} + \overline{B_{41}} & \overline{B_{13}} + \overline{B_{43}} - B_{12} - \overline{B_{21}} & B_{21} + B_{24} - B_{13} - \overline{B_{31}} & -B_{14} - \overline{B_{41}} \\
            \overline{B_{42}} - B_{21} & B_{11} + B_{14} + \overline{B_{14}} + \overline{B_{44}} - B_{22} - \overline{B_{22}} & -B_{23} - \overline{B_{32}} & B_{21} - \overline{B_{42}} \\
            B_{34} - \overline{B_{13}} & -B_{32} - \overline{B_{23}} & B_{41} + B_{44} + \overline{B_{11}} + \overline{B_{41}} - B_{33} - \overline{B_{33}} & \overline{B_{13}} - B_{34} \\
            -B_{41} - \overline{B_{14}} & B_{31} + B_{34} - B_{42} - \overline{B_{24}} & \overline{B_{12}} + \overline{B_{42}} - B_{43} - \overline{B_{34}} & B_{41} + \overline{B_{14}}
        \end{array}\right]
    \end{aligned}
    $ }
    \]
\end{proof}

A rigorous classification of quantum graphs requires an appropriate parametrization of their edge structures. In order to describe families of adjacency operators uniquely and eliminate redundancies caused by graph isomorphisms, we introduce carefully chosen index sets. The parameters $\alpha$ and $\beta$ associated with these sets provide a coordinate system for the resulting parameter space of quantum graphs.

\begin{definition}
    We define the set $J^{(1B)}:=\{(\alpha, \beta): (*)\text{ is fulfilled}\}$, where
    \[ (*)\quad 
    \begin{array}{lcr}
        \beta=1 &\implies &\alpha\in [0, \infty)\\
        \beta\in [0, 1) &\implies& \arg(\alpha) \in [0, \pi)
    \end{array}. \]
\end{definition}

\begin{theorem}\cite[ Theorem 5.2]{Kiefer_2026} Any unit tracial quantum graph on $\Mat_2$ with one quantum edge is isomorphic to either $\left(\Mat_2, 2\Tr, A^{(1A)}\right)$ or exactly one quantum graph of the form $\left(\Mat_2, 2\Tr, A^{(1B)}_{\alpha, \beta}\right)$, where $(\alpha, \beta)\in J^{(1B)}$,
\[ \kappa A^{(1A)}\kappa^{-1}:=\left[\begin{array}{cccc}
    1 & 0 & 0 & 0 \\
    0 & 1 & 0 & 0 \\
    0 & 0 & 1 & 0 \\
    0 & 0 & 0 & 1 \\
\end{array}\right], \quad \kappa A^{(1B)}_{\alpha,\beta}\kappa^{-1}:=c^{-1}\left[\begin{array}{cccc}
    |\alpha|^2 & \alpha \beta_+ & \overline{\alpha}\beta_+ & \beta_+^2 \\
    \alpha \beta_- & |\alpha|^2 & \beta_-\beta_+ & \overline{\alpha}\beta_+ \\
    \overline{\alpha}\beta_- & \beta_-\beta_+ & |\alpha|^2 & \alpha \beta_+ \\
    \beta_-^2 & \overline{\alpha}\beta_- & \alpha \beta_- & |\alpha|^2 \\
\end{array}\right], \]
 $\beta_{\pm}=1\pm \beta$, and $c=|\alpha|^2+1+\beta^2$. Moreover, the graph $\left(\Mat_2, 2\Tr, A^{(1B)}_{\alpha, \beta}\right)$ is:
 \begin{itemize}
     \item undirected if and only if $\alpha\in \RR$ and $\beta=0$,
     \item never reflexive,
     \item loop-free if and only if $\alpha=0$.
 \end{itemize}
 Hence, the only loop-free one-edge quantum graph corresponds to the adjacency matrix $A^{(1B)}_{0,0}$.
\end{theorem}

\begin{example}
    Now we compute the quantum Laplacians of all one-edge quantum graphs. Using Lemma \ref{Laplacian:formula}, we find that $\Delta_{\kappa A^{(1A)}\kappa^{-1}}=0$ and:
    \[ \Delta_{\kappa A^{(1B)}_{\alpha, \beta}\kappa^{-1}}=c^{-1}\left[\begin{array}{cccc}
         1+\beta^2 & 0 & 0 & -(1+\beta^2) \\
         0 & (1+\beta)^2 & -(1-\beta^2) & 0 \\
         0 & -(1-\beta^2) & (1-\beta)^2 & 0 \\
         -(1+\beta^2) & 0 & 0 & 1+\beta^2 \\
    \end{array}\right]. \]
    Finally, $\mathrm{spec}\left(\Delta_{\kappa A^{(1B)}_{\alpha, \beta}\kappa^{-1}}\right)=\left\{0, 0, \frac{2(1+\beta^2)}{c}, \frac{2(1+\beta^2)}{c}\right\}$. Since $0$ is an eigenvalue with multiplicity greater than $1$, the graph is not connected.
\end{example}

In order to classify two-edge unit tracial quantum graphs, we need to define another family of scalars.

\begin{definition}
    We define the set $J^{(2)}:=\{(\beta, \gamma, \delta): (**)\text{ is fulfilled}\}$, where
    \[ (**)\quad\begin{array}{rcl}
        \beta=0, \gamma/\delta \neq \pm i & \implies & \arg(\gamma)\in [0, \pi), \arg(\delta)=\arg(\gamma)+\pi, |\delta|\leq |\gamma|,\\  
         \beta=0, \gamma/\delta = \pm i & \implies & \gamma\in [0, \infty), \arg(\delta)=\arg(\gamma)+\pi, |\delta|\leq |\gamma|,\\ 
         \beta\in(0, 1) & \implies & \arg(\gamma)\in[0, \pi),\\
         \beta=1 & \implies & \gamma\in [0, \pi).
    \end{array} \]   
\end{definition}

\begin{theorem}\cite[Proposition 6.3]{Kiefer_2026} 
Let $(\beta, \gamma, \delta)\in J^{(2)}$. Denote $\delta_\pm=1\pm\delta$ and $\beta_\pm=1\pm\beta$. Any unit tracial quantum graph on $\Mat_2$ with two quantum edges is isomorphic to exactly one of the quantum graphs $\mathcal{G}_{\beta,\gamma,\delta}^{(2)}:=\left(\Mat_2, 2\Tr, A^{(2)}_{\beta, \gamma, \delta}\right)$ or $\mathcal{G}^{(2c)}_{\beta, \gamma, \delta}:=(\Mat_2, 2\Tr, A_c-A^{(2)}_{\beta, \gamma, \delta})$, where
 \begin{align*}
        A_{\beta,\gamma,\delta}^{(2)}=
        &c_1^{-1}
        \left[\begin{array}{cccc}
            0&0&0&\beta_+^2\\
            0&0&\beta_-\beta_+&0\\
            0&\beta_-\beta_+&0&0\\
            \beta_-^2&0&0&0
        \end{array}\right]\\
        &+
        c_2^{-1}
            \left[\begin{array}{cccc}
                |\delta_+|^2&i\overline{\gamma}\beta_-\delta_+&-i\gamma\beta_-\overline{\delta}_+&|\gamma|^2\beta_-^2\\
                -i\overline{\gamma}\beta_+\delta_+&\delta_+\overline{\delta}_-&-|\gamma|^2\beta_-\beta_+&-i\gamma\beta_-\overline{\delta}_-\\
                i\gamma\beta_+\overline{\delta}_+&-|\gamma|^2\beta_-\beta_+&\overline{\delta}_+\delta_-&i\overline{\gamma}\beta_-\delta_-\\
                |\gamma|^2\beta_+^2&i\gamma\beta_+\overline{\delta}_-&-i\overline{\gamma}\beta_+\delta_-&|\delta_-|^2
            \end{array}\right]
    \end{align*}    
        $c_1\vcentcolon=1+\beta^2$ and $c_2\vcentcolon=1+|\delta|^2+|\gamma|^2(1+\beta^2)$.
    The quantum graph $\mathcal{G}_{\beta,\gamma,\delta}^{(2)}$ is:
    \begin{itemize}
        \item undirected if and only if $\beta=0$ and $\gamma,\delta\in\RR$, 
        \item reflexive if and only if $\delta=\gamma=0$,
        \item never loop-free.
    \end{itemize}    
\end{theorem}

\begin{example}
    In this example, we compute the quantum Laplacians of all two-edge tracial unit quantum graphs. For any $(\beta, \gamma, \delta)\in J^{(2)}$, we evaluate:
    \[ \Delta_{\kappa A^{(2)}_{\beta, \gamma, \delta}\kappa^{-1}}=\left[\begin{array}{cccc}
1 + \frac{|\gamma|^2 c_1}{c_2} & -i\frac{\gamma\overline{\delta}\beta_+}{c_2} & i\frac{\gamma\overline{\delta}\beta_-}{c_2} & -\left(1 + \frac{|\gamma|^2 c_1}{c_2}\right) \\
i\frac{\overline{\gamma}\delta\beta_+}{c_2} & \frac{2|\delta|^2}{c_2} + \frac{\beta_+^2}{c_1} + \frac{|\gamma|^2\beta_-^2}{c_2} & -\frac{\beta_-\beta_+}{c_1} + \frac{|\gamma|^2\beta_-\beta_+}{c_2} & -i\frac{\overline{\gamma}\delta\beta_+}{c_2} \\
-i\frac{\overline{\gamma}\delta\beta_-}{c_2} & -\frac{\beta_-\beta_+}{c_1} + \frac{|\gamma|^2\beta_-\beta_+}{c_2} & \frac{2|\delta|^2}{c_2} + \frac{\beta_-^2}{c_1} + \frac{|\gamma|^2\beta_+^2}{c_2} & i\frac{\overline{\gamma}\delta\beta_-}{c_2} \\
-\left(1 + \frac{|\gamma|^2 c_1}{c_2}\right) & i\frac{\gamma\overline{\delta}\beta_+}{c_2} & -i\frac{\gamma\overline{\delta}\beta_-}{c_2} & 1 + \frac{|\gamma|^2 c_1}{c_2}
\end{array}\right]. \]
The spectrum of this matrix is $\mathrm{spec}\left(\Delta_{\kappa A^{(2)}_{\beta, \gamma, \delta}\kappa^{-1}}\right)=\left\{0, \frac{2(|\delta|^2+c_1|\gamma|^2)}{c_2},2,2+ \frac{2(|\delta|^2+c_1|\gamma|^2)}{c_2}\right\}$. Note that this graph is disconnected if and only if 
\[ 0=|\delta|^2+|\gamma|^2(1+\beta^2). \] In particular, the only disconnected two-edge unit tracial graph is $\mathcal{G}^{(2)}_{\beta,0,0}$.
\end{example}

After classifying quantum graphs with one and two edges, the three-edge case could in principle be treated by repeating the previous analysis. Nevertheless, a more efficient approach follows from the complement operation. The dimension of the underlying space $\mathrm{Mat}_2$ implies that a quantum graph can contain at most four mutually orthogonal edges. Therefore, once the complete quantum graph with four edges has been characterized, every three-edge graph can be constructed as the complement of a suitable one-edge graph.

Let us recall that for any quantum system $(\mathcal{M},\Psi)$, there is a canonical complete quantum graph associated with the adjacency operator
\[
A_c:=\Psi(\cdot)\I.
\]
In the setting of unital tracial quantum graphs over $\mathrm{Mat}_2$, the dimension of the underlying algebra is $4$, implying that the complete quantum graph contains exactly four quantum edges. We denote its adjacency operator by
\[
A^{(4)}:=2\operatorname{Tr}(\cdot)\I.
\]
A direct computation shows that its matrix representation with respect to the isometry $\kappa$ is
\[
\kappa A^{(4)}\kappa^{-1}
=
\begin{bmatrix}
2&0&0&2\\
0&0&0&0\\
0&0&0&0\\
2&0&0&2
\end{bmatrix}.
\]

Using Lemma~\ref{graph_complement}, we can systematically construct new quantum graphs on $\mathrm{Mat}_2$ by taking complements with respect to this complete quantum graph. Since the maximal number of mutually orthogonal quantum edges in this setting is four, the complement of any one-edge quantum graph is necessarily a three-edge quantum graph. This duality provides an efficient method for classifying all three-edge tracial unital quantum graphs: they are obtained simply by subtracting the adjacency operators of the one-edge graphs from the complete adjacency operator $A^{(4)}$.

\begin{example}
    Since all tracial unit quantum graphs with three edges are the graph complements of one-edge tracial unit quantum graphs, the adjacency matrices of such graphs are obtained naturally by taking $\kappa A^{(4)}\kappa^{-1} - \kappa A^{(1)}\kappa^{-1}$. Their adjacency matrices are given below:
    \[ \kappa A^{(3A)}\kappa^{-1}=\left[\begin{array}{cccc}
        1 & 0 & 0 & 2 \\
        0 & -1 & 0 & 0 \\
        0 & 0 & -1 & 0 \\
        2 & 0 & 0 & 1 \\
    \end{array}\right] \]
    and 
    \[ \kappa A^{(3B)}_{\alpha, \beta}\kappa^{-1} =c^{-1} \begin{bmatrix}
|\alpha|^2 + \beta_+^2 + \beta_-^2 & -\alpha \beta_+ & -\overline{\alpha}\beta_+ & 2|\alpha|^2 + \beta_-^2 \\
-\alpha \beta_- & -|\alpha|^2 & -\beta_-\beta_+ & -\overline{\alpha}\beta_+ \\
-\overline{\alpha}\beta_- & -\beta_-\beta_+ & -|\alpha|^2 & -\alpha \beta_+ \\
2|\alpha|^2 + \beta_+^2 & -\overline{\alpha}\beta_- & -\alpha \beta_- & |\alpha|^2 + \beta_+^2 + \beta_-^2
\end{bmatrix}, \]
where $(\alpha, \beta)\in J^{(1B)}$, $\beta_\pm=1\pm\beta$, and $c=|\alpha|^2+\beta^2+1$. Substituting this into the formula from Lemma \ref{Laplacian:formula}, we obtain:
\begin{equation}\label{A(3A):Laplacian}
\Delta_{\kappa A^{(3A)}\kappa^{-1}}=\left[\begin{array}{cccc}
    2 & 0 & 0 & -2 \\
    0 & 4 & 0 & 0 \\
    0 & 0 & 4 & 0\\
    -2 & 0 & 0 & 2 \\
\end{array}\right]
\end{equation}
and 
\[ \Delta_{\kappa A^{(3B)}_{\alpha, \beta}\kappa^{-1}}= \begin{bmatrix}
1 + \frac{|\alpha|^2}{c} & 0 & 0 & -\left(1 + \frac{|\alpha|^2}{c}\right) \\
0 & 3 + \frac{|\alpha|^2 - 2\beta}{c} & \frac{1-\beta^2}{c} & 0 \\
0 & \frac{1-\beta^2}{c} & 3 + \frac{|\alpha|^2 + 2\beta}{c} & 0 \\
-\left(1 + \frac{|\alpha|^2}{c}\right) & 0 & 0 & 1 + \frac{|\alpha|^2}{c}
\end{bmatrix}. \]
Moreover, $\mathrm{spec}\left(\Delta_{\kappa A^{(3A)}\kappa^{-1}}\right)=\{0, 4,4,4\}$ and $\mathrm{spec}\Bigl(\Delta_{\kappa A^{(3B)}_{\alpha, \beta}\kappa^{-1}}\Bigr)=\left\{0, \, 2+\frac{2\alpha^2}{c}, \, 2+\frac{2\alpha^2}{c}, \, 4\right\}$. In particular, we observe that these graphs are always connected.
\end{example}

\begin{example}
    In this example, we will compute the Laplacian of the complete unital tracial quantum graph on $\Mat_2$. We use the fact that
    \[ \Delta_{\kappa A^{(4)} \kappa^{-1}}=\Delta_{\kappa( A^{(3A)}+A^{(1A)} )\kappa^{-1}}=\Delta_{\kappa A^{(3A)} \kappa^{-1}}+\Delta_{\kappa A^{(1A)} \kappa^{-1}}=\Delta_{\kappa A^{(3A)} \kappa^{-1}}, \]
    which implies that this Laplacian is exactly given by formula \eqref{A(3A):Laplacian}. Note that because $A^{(4)}$ and $A^{(3A)}$ differ only by loops (as $A^{(1A)}$ represents a single loop), their Laplacians must be identical.
\end{example}

\begin{example}
    Using the linearity of the Laplacian again, we can easily compute the Laplacian of the complement graph $\mathcal{G}^{(2c)}_{\beta, \gamma, \delta}$:
    \[ \Delta_{\kappa A^{(2c)}\kappa^{-1}}=\Delta_{\kappa A^{(4)}\kappa^{-1}}-\Delta_{\kappa A^{(2)}_{\beta, \gamma, \delta}\kappa^{-1}}. \]
    Substituting the matrices calculated in the previous examples, we obtain:
    \[
    \resizebox{\textwidth}{!}{ $
    \Delta_{\kappa A^{(2c)}\kappa^{-1}} = \left[\begin{array}{cccc}
        \frac{1+|\delta|^2}{c_2} & i\frac{\gamma\overline{\delta}\beta_+}{c_2} & -i\frac{\gamma\overline{\delta}\beta_-}{c_2} & -\frac{1+|\delta|^2}{c_2} \\
        -i\frac{\overline{\gamma}\delta\beta_+}{c_2} & 4 - \frac{2|\delta|^2}{c_2} - \frac{\beta_+^2}{c_1} - \frac{|\gamma|^2\beta_-^2}{c_2} & \frac{\beta_-\beta_+}{c_1} - \frac{|\gamma|^2\beta_-\beta_+}{c_2} & i\frac{\overline{\gamma}\delta\beta_+}{c_2} \\
        i\frac{\overline{\gamma}\delta\beta_-}{c_2} & \frac{\beta_-\beta_+}{c_1} - \frac{|\gamma|^2\beta_-\beta_+}{c_2} & 4 - \frac{2|\delta|^2}{c_2} - \frac{\beta_-^2}{c_1} - \frac{|\gamma|^2\beta_+^2}{c_2} & -i\frac{\overline{\gamma}\delta\beta_-}{c_2} \\
        -\frac{1+|\delta|^2}{c_2} & -i\frac{\gamma\overline{\delta}\beta_+}{c_2} & i\frac{\gamma\overline{\delta}\beta_-}{c_2} & \frac{1+|\delta|^2}{c_2}
    \end{array}\right].
    $ }
    \]

    Moreover, since the complete graph Laplacian $\Delta_{\kappa A^{(4)}\kappa^{-1}}$ has eigenvalues $\{0, 4, 4, 4\}$ and shares the null space with $\Delta_{\kappa A^{(2)}_{\beta, \gamma, \delta}\kappa^{-1}}$, we can determine the spectrum of the complement graph directly. Knowing that $\mathrm{spec}\Bigl(\Delta_{\kappa A^{(2)}_{\beta, \gamma, \delta}\kappa^{-1}}\Bigr)=\left\{0, \frac{2(|\delta|^2+c_1|\gamma|^2)}{c_2},2,2+ \frac{2(|\delta|^2+c_1|\gamma|^2)}{c_2}\right\}$, the spectrum of the complement graph is simply derived from the difference: $\left\{0, 4-\frac{2(|\delta|^2+c_1|\gamma|^2)}{c_2}, 4-2, 4-\left(2+\frac{2(|\delta|^2+c_1|\gamma|^2)}{c_2}\right)\right\}$. This simplifies neatly to:
    \[ \mathrm{spec}\left(\Delta_{\kappa A^{(2c)}\kappa^{-1}}\right) = \left\{0, 2, \frac{2}{c_2}, 2+\frac{2}{c_2}\right\}. \]
    Notice that since $c_2 \geq 1$, the algebraic multiplicity of the eigenvalue $0$ is always exactly $1$. Consequently, the complement graph $\mathcal{G}^{(2c)}_{\beta, \gamma, \delta}$ is always connected.
\end{example}

\section*{Acknowledgments}
The study was funded by ``Laboratories of the Young'' as part of the ``Excellence Initiative – Research University'' program at Jagiellonian University in Kraków (LM/44/BA). AB is supported by the Alexander von Humboldt Foundation. The work of AB was partially funded by the Deutsche Forschungsgemeinschaft (DFG, German Research Foundation) under Germany’s Excellence Strategy - EXC-2111 – 390814868. This is a part of the international project Graph Algebras partially supported by EU grant HORIZON-MSCA SE-2021 Project 101086394 and co-financed by the Polish Ministry of Education and Science within the program PMW under contract 5447/HE/2023/2.

\bibliography{biblproj}{}
\bibliographystyle{plain}
\end{document}